\documentclass{article}

\usepackage[centertags]{amsmath}
\usepackage{hyperref}
\usepackage{amsfonts}
\usepackage{amssymb}
\usepackage{amsthm}
\usepackage{newlfont}
\usepackage{amscd}
\usepackage{amsmath,amscd}

\usepackage[curve,arrow,cmactex,matrix]{xy}
\usepackage{verbatim}

\usepackage{color}

\usepackage{eucal}
\usepackage{graphicx}

\newcommand{\CC}{\mathbb{C}}
\newcommand{\QQ}{\mathbb{Q}}
\newcommand{\ZZ}{\mathbb{Z}}
\newcommand{\A}{\mathbb{A}}
\newcommand{\PP}{\mathbb{P}}

\newcommand{\RR}{\mathbb{R}}

\newcommand{\FF}{\mathbb{F}}
\newcommand{\GG}{\mathbb{G}}
\newcommand{\DD}{\mathbb{D}}

\newcommand{\C}{\mathcal{C}}

\newcommand{\U}{\mathcal{U}}

\newcommand{\X}{\mathcal{X}}

\newcommand{\OO}{\mathcal{O}}

\newtheorem{thm}{Theorem}[subsection]
\newtheorem{cor}[thm]{Corollary}
\newtheorem{conj}[thm]{Conjecture}
\newtheorem{lem}[thm]{Lemma}
\newtheorem{prop}[thm]{Proposition}

\newtheorem{question}[thm]{Question}

\theoremstyle{definition}
\newtheorem{define}[thm]{Definition}
\newtheorem{const}[thm]{Construction}

\theoremstyle{remark}
\newtheorem{rem}[thm]{Remark}
\newtheorem{example}[thm]{Example}
\DeclareMathOperator{\Gal}{Gal}

\DeclareMathOperator{\Pic}{Pic}

\DeclareMathOperator{\Br}{Br}

\DeclareMathOperator{\Log}{Log}

\DeclareMathOperator{\ev}{ev}

\DeclareMathOperator{\Hom}{Hom}

\DeclareMathOperator{\df}{def}

\DeclareMathOperator{\spec}{spec}
\DeclareMathOperator{\rank}{rank}
\DeclareMathOperator{\Aut}{Aut}

\DeclareMathOperator{\et}{\acute{e}t}

\def\alp{{\alpha}}

\def\sig{{\sigma}}
\def\vphi{{\varphi}}

\def\Gam{{\Gamma}}
\def\Del{{\Delta}}

\def\Lam{{\Lambda}}

\def\vphi{{\varphi}}

\def\lrar{\longrightarrow}

\def\hrar{\hookrightarrow}

\def\x{\stackrel}

\def\ovl{\overline}

\def\wtl{\widetilde}
\def\bksl{\;\backslash\;}

\title{Geometry and arithmetic of certain log K3 surfaces}

\author{Yonatan Harpaz}

\begin{document}
\maketitle

\begin{abstract}
Let $k$ be a field of characteristic $0$. In this paper we describe a classification of smooth log K3 surfaces $X$ over $k$ whose geometric Picard group is trivial and which can be compactified into del Pezzo surfaces. We show that such an $X$ can always be compactified into a del Pezzo surface of degree $5$, with a compactifying divisor $D$ being a cycle of five $(-1)$-curves, and that $X$ is completely determined by the action of the absolute Galois group of $k$ on the dual graph of $D$. When $k=\QQ$ and the Galois action is trivial, we prove that for any integral model $\X/\ZZ$ of $X$, the set of integral points $\X(\ZZ)$ is not Zariski dense. We also show that the Brauer Manin obstruction is not the only obstruction for the integral Hasse principle on such log K3 surfaces, even when their compactification is ``split''.
\end{abstract}

\tableofcontents

\section{Introduction}

Let $k$ be a number field, $S$ a finite set of places and $\OO_S$ the ring of $S$-integers in $k$. By an $\OO_S$-scheme we mean a separated scheme of finite type over $\OO_S$. A fundamental question in number theory is to understand the set $\X(\OO_S)$ of $S$-integral points. In this paper we shall be interested in the case where $X = \X \otimes_{\OO_S} k$ is a smooth \textbf{log K3 surface}. 

When studying integral points on schemes, the following two questions are of great interest.
\begin{enumerate}
\item
Given an $\OO_S$-scheme $\X$, is the set $\X(\OO_S)$ non-empty?
\item
If the set $\X(\OO_S)$ is non-empty, is it in any sense ``large''?
\end{enumerate}

Let us begin with the second question. Consider the question of counting points of bounded height with respect to some height function. Informally speaking, the behavior of integral points on smooth log K3 surfaces is expected to parallel the behaviour of rational points on smooth and proper K3 surfaces. The following is one of the variants of a conjecture appearing in~\cite{VL}:

\begin{conj}[\cite{VL}]\label{c:vl}
Let $X$ be a K3 surface over a number field $k$, and let $H$ be a height function associated to an ample divisor. Suppose that $X$ has geometric Picard number $1$. Then there exists a Zariski open subset $U \subseteq X$ such that 
$$ \#\{P \in U(k) | H(P) \leq B\} = O(\log(B)) .$$
\end{conj}

When considering integral points on smooth log K3 surfaces, one might expect to obtain a similar logarithmic estimate for the growth of \textbf{integral points}. We note that the minimal geometric Picard number a non-proper smooth log K3 surface may attain is $0$ (although one might argue that this is by no means the ``generic'' case, see Remark~\ref{r:rigid}). We then consider the following conjecture:

\begin{conj}\label{c:pic-0}
Let $k$ be a number field and $S$ a finite set of places of $k$ containing the infinite places. Let $\X$ be a smooth, separated scheme over $\OO_S$ such that $X = \X \otimes_{k} X$ is a log K3 surface with $\Pic(X \otimes_{k} \ovl{k}) = 0$. Let $H$ be a height function associated to an ample divisor. Then there exists a Zariski open subset $\U \subseteq \X$ and a constant $b$ such that
$$ \#\{P \in \U(\OO_k) | H(P) \leq B\} \simeq O\left(\log(B)^b\right). $$
\end{conj}

Our main goal in this paper is to give evidence for this conjecture. We focus our attention on what can be considered as the simplest class of log K3 surfaces, namely, those whose log K3 structure comes from a compactification into a del Pezzo surface. We call such log K3 surfaces \textbf{ample log K3 surfaces}.

The bulk of this paper is devoted to the classification of log K3 surfaces of this ``simple'' kind over a general base field $k$ of characteristic $0$, under the additional assumption that $\Pic(X \otimes_k \ovl{k}) = 0$. Our first result is that such surfaces can always be compactified into a del Pezzo surface $\ovl{X}$ of degree $5$, with a compactifying divisor $D = \ovl{X} \bksl X$ being a cycle of five $(-1)$-curve. The Galois action on the dual graph of $D$ then yields an invariant $\alp \in H^1(k,\DD_5)$, where $\DD_5$ is the dehidral group of order 10, considered here as the automorphism group of a cyclic graph of length $5$. We then obtain the following classification theorem:
\begin{thm}[See Theorem~\ref{t:well-defined} and Theorem~\ref{t:main-class-2}]\label{t:main-class}
The element $\alp = \alp_X \in H^1(k,\DD_5)$ does not depend on the choice of $\ovl{X}$. Furthermore, the association $X \mapsto \alp_X$ determines a bijection between the set of $k$-isomorphism classes of ample log K3 surfaces of Picard rank $0$ and the Galois cohomology set $H^1(k,\DD_5)$.
\end{thm}
%
%
%We then show that the action of the absolute Galois group of $k$ on the dual graph of $D$ determines $X$ up to \textbf{isomorphism} over $k$. Furthermore, any possible Galois action can be realized, yielding a complete classification. 

Our second main result is a verification of Conjecture~\ref{c:pic-0} for ample log K3 surface of Picard rank $0$ corresponding to the trivial class in $H^1(k,\DD_5)$. More precisely, we have the following
\begin{thm}[See Theorem~\ref{t:zariski}]\label{t:zariski-intro}
Let $\X$ be a separated smooth scheme of finite type over $\ZZ$ such that $X = \X \otimes_{\ZZ} \QQ$ is an ample log K3 surface of Picard rank $0$ and such that $\alp_X = 0$. Then the set of integral points $\X(\ZZ)$ is not Zariski dense.
\end{thm}
%
% that when $k=\QQ$, $S = \{\infty\}$ and the Galois action on the dual graph of $D$ is trivial (this corresponds to a unique log K3 surface $X$ over $\QQ$, which can be given explicit equations), the set of integral points on any integral model for $X$ is not Zariski dense. This verifies Conjecture~\ref{c:pic-0} in this case with $b=0$. 
% 
We also give an example with the same $k=\QQ$ and $S=\{\infty\}$, in which $\alp_X \neq 0$, and where integral points are in fact Zariski dense. In particular, if Conjecture~\ref{c:pic-0} holds in this case then $b$ must be at least $1$. We would be very interested to understand what properties of $X$ control the value of $b$, when such a value exists.

Now consider Question (1) above, namely, the \textbf{existence} of integral points. The study of this question often begins by considering the set of $S$-integral adelic points 
$$ \X(\A_{k,S}) \x{\df}{=} \prod_{v \in S}X(k_v) \times \prod_{v \notin S} \X(\OO_v) $$ 
where $X = \X \otimes_{\OO_S} k$ is the base change of $\X$ to $k$. If $\X(\A_{k,S}) = \emptyset$ one may immediately deduce that $X$ has no $S$-integral points. In general, it can happen that $\X(\A_{k,S}) \neq \emptyset$ but $\X(\OO_S)$ is still empty. One way to account for this phenomenon is given by the integral version of the \textbf{Brauer-Manin obstruction}, introduced in~\cite{CTX09}. To this end one considers the set
$$ \X(\A_{k,S})^{\Br(X)} \x{\df}{=} \X(\A_{k,S}) \cap X(\A_k)^{\Br(X)} $$
given by intersecting the set of $S$-integral adelic points with the Brauer set of $X$. When $\X(\A_{k,S})^{\Br(X)} = \emptyset$ one says that there is a Brauer-Manin obstruction to the existence of $S$-integral points.

\begin{question}\label{q:BM}
Is there a natural class of $\OO_S$-schemes for which we should expect the integral Brauer-Manin obstruction to be the only one?
\end{question}

If $X$ is proper then the set of $S$-integral points on $\X$ can be identified with the set of \textbf{rational points} on $X$. In this case, the class of smooth and proper \textbf{rationally connected varieties} is conjectured to be a natural class where the answer to Question~\ref{q:BM} is positive (see~\cite{CT03}). If $X$ is not proper, then the question becomes considerably more subtle. One may replace the class of rationally connected varieties by the class of (smooth) \textbf{log rationally connected varieties} (see~\cite{Zh14}). However, even for log rationally connected varieties, the integral Brauer-Manin obstruction is not the only one in general. One possible problem is that log rationally connected varieties may have a non-trivial (yet always finite) fundamental group. In~\cite[Example 5.10]{CTW12} Colliot-Th\'el\`ene and Wittenberg consider a log rationally connected surface $\X$ over $\spec(\ZZ)$ with fundamental group $\ZZ/2$. It is then shown that the Brauer-Manin obstruction is trivial, but one can prove that integral points do not exist by applying the Brauer-Manin obstruction to the universal covering of $X$. This can be considered as a version of the \'etale-Brauer Manin obstruction for integral points. In~\cite[Example 5.9]{CTW12}, Colliot-Th\'el\`ene and Wittenberg consider another surface $\X$ over $\spec(\ZZ)$ given by the equation
\begin{equation}\label{e:CTW}
 2x^2 + 3y^2 + 4z^2 = 1 .
\end{equation}
This time $\X$ is simply-connected and log rationally connected, and yet $\X(\ZZ) = \emptyset$ with no integral Brauer-Manin obstruction. This unsatisfactory situation motivates the following definition:

\begin{define}\label{d:split}
Let $X$ be a smooth, geometrically integral variety over a field $k$ and let $X \subseteq \ovl{X}$ be a smooth compactification such that $D = \ovl{X} \bksl X$ is a simple normal crossing divisor. Let $\C_{\ovl{k}}(D)$ be the \textbf{geometric Clemens complex} of $D$ (see~\cite[\S 3.1.3]{CT10}). We will say that the compactification $(\ovl{X},D)$ is \textbf{split} (over $k$) if the Galois action on $\C_{\ovl{k}}(D)$ is trivial and for every point in $\C_{\ovl{k}}(D)$ the corresponding subvariety of $D$ has a $k$-point. We will say that $X$ itself is split if it admits a split compactification.
\end{define}

\begin{example}\label{ex:quadratic}
Let $k=\RR$ and let $f(x_1,...,x_n)$ be a non-degenerate quadratic form. The affine variety $X$ given by
\begin{equation}\label{e:f}
f(x_1,...,x_n) = 1 
\end{equation}
is split if and only if the equation $f(x_1,...,x_n) = 0$ is soluble in $k$. Indeed, if $f$ represents $0$ over $k$ then we have a split compactification $\ovl{X} \subseteq \PP^n$ given by $f(x_1,...,x_n) = y^2$. On the other hand, if $f$ does not represent $0$ then the space $X(k)$ is compact (with respect to the real topology) and hence $X$ cannot admit any split compactification. A similar argument works when $k$ is any local field.
\end{example}

\begin{define}\label{d:S-split}
Let $\X$ be an $\OO_S$-scheme. We will say that $\X$ is \textbf{$S$-split} if $\X \otimes_{\OO_v} k_v$ is split over $k_v$ for at least one $v \in S$.
\end{define}

The class of $S$-split schemes appears to be more well-behaved with respect to $S$-integral points. As example~\ref{ex:quadratic} shows, counter-example~\ref{e:CTW} is indeed not $S$-split. 

\begin{question}\label{q:BM-S}
Is the integral Brauer-Manin obstruction the only obstruction for smooth, $S$-split, simply-connected $\OO_S$-schemes?
\end{question}

Our last goal in this paper is to show that the answer to this question is negative. More precisely:
\begin{thm}[See \S\ref{ss:BM}]\label{t:counter}
Let $S = \{\infty\}$. Then there exists an $S$-split log K3 surface over $\ZZ$ for which $X(\ZZ) = \emptyset$ and yet the integral Brauer-Manin obstruction is trivial. 
\end{thm}

To the knowledge of the author this is the first example of a smooth, $S$-split and simply-connected $\OO_S$-scheme for which the Brauer-Manin obstruction is shown to be insufficient.

Despite Theorem~\ref{t:counter}, we do believe that the answer to Question~\ref{q:BM-S} is positive when one restricts attention to log rationally connected schemes. More specifically, we offer the following integral analogue of the conjecture of Colliot-Th\'el\`ene and Sansuc (see~\cite{CT03}):
\begin{conj}\label{c:rationally-connected}
Let $\X$ be an $S$-split smooth $\OO_S$-scheme such that $X = \X \otimes_{\OO_S} k$ is a simply connected, log rationally connected variety. If $\X(\A_{k,S})^{\Br(X)} \neq \emptyset$ then $X(\OO_S) \neq \emptyset$.
\end{conj}

\section{Preliminaries}

Let $k$ be a field of characteristic $0$. We will denote by $\ovl{k}$ a fixed algebraic closure of $k$ and by $\Gam_k = \Gal(\ovl{k}/k)$ the absolute Galois group of $k$. Given a variety $X$ over $k$ and a field extension $L/k$, we will denote by $X_L = X \otimes_k L$ the base change of $X$ to $L$. Let us begin with some basic definitions.

\begin{define}
Let $Z$ be a separated scheme of finite type over $k$. By a \textbf{geometric component of $Z$} we mean an irreducible component of $Z_{\ovl{k}}$. The set $C(Z)$ of geometric components of $Z$ is equipped with a natural action of $\Gam_k$ and we will often consider it as a Galois set. If $Z \subset Y$ is an effective divisor of smooth variety $Y$ then we will say that $Z$ has \textbf{simple normal crossing} if each geometric component of $Z$ is smooth and for every subset $I \subseteq C(Z)$ the intersection of $\{D_i\}_{i \in I}$ is either empty or pure of dimension $\dim(Z) - |I| + 1$ and transverse at every point. If $\dim(Z) = 1$ this means that each component of $Z$ is a smooth curve and each intersection point is a transverse intersection of exactly two components.
\end{define}

\begin{define}\label{d:dual}
Let $Y$ be a smooth surface over $k$ and $D \subseteq Y$ a simple normal crossing divisor. The \textbf{dual graph} of $D$, denote $G(D)$, is the graph whose vertices are the geometric components of $D$ and whose edges are the intersection points of the various components (since $D$ is a simple normal crossing divisor no self intersections or triple intersections are allowed). Note that two distinct vertices may be connected by more than one edge. We define the \textbf{splitting field} of $D$ is the minimal extension $L/k$ such that the action of $\Gam_L$ on $G(D)$ is trivial. In other words, the minimal extension over which all components and all intersection points are defined.
\end{define}
%
%\begin{define}\label{d:splitting}
%Let $Y$ be a smooth surface over $k$ and $D \subseteq Y$ a simple normal crossing divisor. 
%\end{define}

\begin{define}\label{d:normal}
Let $X$ be a smooth geometrically integral surface over $k$. A \textbf{simple compactification} of $X$ is a smooth compactification $\iota: X \hrar \ovl{X}$ (defined over $k$) such that $D = \ovl{X} \bksl X$ is a simple normal crossing divisor. 
\end{define}

%In this paper we will be mostly interested in the following types of surfaces:
\begin{define}\label{d:k3}
Let $X$ be a smooth geometrically integral surface over $k$. A \textbf{log K3 structure} on $X$ is a simple compactification $(X,D,\iota)$ such that $$[D] = -K_{\ovl{X}} $$ (where $K_{\ovl{X}} \in \Pic(X)$ is the canonical class of $X$). A \textbf{log K3 surface} is a smooth, geometrically integral, \textbf{simply connected} surface $X$ equipped with a log K3 structure $(\ovl{X},D, \iota)$.
\end{define}

\begin{rem}
The property of being simply connected in Definition~\ref{d:k3} is intended in the geometric since, i.e., the \'etale fundamental group of the base change $X_{\ovl{k}}$ is trivial.
\end{rem}

\begin{rem}
Let $X$ be a log K3 surface. Then it is not true that every simple compactification $\ovl{X}$ of $X$ is a log K3 structure. However, every simple compactification $(X,D,\iota)$ satisfies the weaker property that $\dim H^0(X,n([D]+K_X)) = 1$ for all $n \geq 0$. This latter property is sometimes taken instead of $[D] = -K_{\ovl{X}} = 0$ in the definition of a log K3 surface.
\end{rem}

\begin{prop}\label{p:rational}
Let $X$ be a log K3 surface and $(\ovl{X},D,\iota)$ a log K3 structure on $X$. Then either $D = \emptyset$ and $X = \ovl{X}$ is a (proper) K3 surface or $D \neq \emptyset$ and $\ovl{X}_{\ovl{k}}$ is a rational surface.
\end{prop}
\begin{proof}
If $D = \emptyset$ then $X = \ovl{X}$ is proper and $K_{X} = K_{\ovl{X}} = 0$. Since $X$ is also simply connected it is a K3 surface. Now assume that $D \neq \emptyset$. Since $-K_X = [D]$ is effective and $\ovl{X}$ is smooth and proper it follows that $H^0(\ovl{X},mK_{\ovl{X}}) = 0$ for every $m \geq 1$, and hence $\ovl{X}$ has Kodaira dimension $-\infty$. In light of Castelnuevo's rationality criterion it will suffice to prove that $H^1(\ovl{X}_{\ovl{k}},\ZZ/n) = 0$ for every $n$. Since $X$ is assumed to be simply connected it will be enough to show that the map $H^1(\ovl{X}_{\ovl{k}},\ZZ/n) \lrar H^1(X_{\ovl{k}},\ZZ/n)$ is injective. This, in turn, follows from the fact that $H^1_{D}(\ovl{X}_{\ovl{k}},\ZZ/n) = 0$ by the purity theorem.
\end{proof}

Proposition~\ref{p:rational} motivates the following definition:
\begin{define}
We will say that a log K3 structure $(\ovl{X},D,\iota)$ is \textbf{ample} if $[D] \in \Pic(\ovl{X})$ is \textbf{ample}. Note that in this case $\ovl{X}$ is a \textbf{del Pezzo surface}. We will say that a log K3 surface is \textbf{ample} if it admits an ample log K3 structure.
\end{define}

\begin{rem}
Any ample log K3 surface is affine, as it is the complement of an ample divisor. This observation can be used to show that not all (non-proper) log K3 surfaces are ample. For example, if $\ovl{X} \lrar \PP^1$ is a rational elliptic surface (with a section) and $D \subseteq \ovl{X}$ is the fiber over $\infty \in \PP^1(k)$ then $X = \ovl{X} \bksl D$ is a log K3 surface admitting a fibration $f:X \lrar \A^1$ into elliptic curves. As a result, any regular function on $X$ is constant along the fibers of $f$ and hence $X$ cannot be affine.
\end{rem}

Now let $X$ be a log K3 surface. Since $X_{\ovl{k}}$ is simply connected it follows that $k^*[X] = k^*$ and that $\Pic(X_{\ovl{k}})$ is torsion free, hence isomorphic to $\ZZ^r$ for some $r$. The integer $r$ is called the \textbf{Picard rank} of $X$. We note that this invariant is sometimes called the \textbf{geometric} Picard number, but since it will be the only Picard rank under consideration, we opted to drop the adjective ``geometric''. The following observation will be useful in identifying log K3 surfaces of Picard rank $0$:
\begin{lem}\label{l:basis}
Let $X$ be a smooth, geometrically integral variety over $k$ and $(\ovl{X},D,\iota)$ a simple compactification of $X$. Then the following conditions are equivalent:
\begin{enumerate}
\item
$\Pic(X_{\ovl{k}}) = 0$ and $\ovl{k}^*[X] = \ovl{k}^*$.
\item
The geometric components of $D$ form a basis for $\Pic(\ovl{X}_{\ovl{k}})$.
\end{enumerate}
\end{lem}
\begin{proof}
Let $V$ be the free abelian group generated by the geometric components of $D$. We may then identify $H^1_D(Y,\GG_m)$ with $V$,  yielding an exact sequence
$$ 0 \lrar \ovl{k}^*[\ovl{X}] \lrar \ovl{k}^*[X] \lrar V \lrar \Pic(\ovl{X}_{\ovl{k}}) \lrar \Pic(X_{\ovl{k}}) \lrar 0 .$$
Since $\ovl{X}$ is proper we have $\ovl{k}^*[\ovl{X}] = \ovl{k}^*$. We hence see that condition (1) above is equivalent to the map $V \lrar \Pic(\ovl{X}_{\ovl{k}})$ being an isomorphism.
\end{proof}

\begin{rem}\label{r:rigid}
Let $X$ be a rational surface and $D \subseteq X$ a simple normal crossing divisor such that $[D] = -K_X$ and such that the components of $D$ are of genus $0$ and form a basis for $\Pic(X_{\ovl{k}})$. By~\cite[Lemma 3.2]{Fr15} the pair $(X,D)$ is \textbf{rigid}, i.e., does not posses any first order deformations (nor any continuous families of automorphisms). By Lemma~\ref{l:basis} we may hence expect that any type of moduli space of log K3 surfaces of Picard rank $0$ will be $0$-dimensional. In particular, the classification problem for such surfaces over a non-algebraically closed field naturally leads to various Galois cohomology sets with coefficients in finite groups.
\end{rem}

We shall now describe two examples of ample log K3 surfaces with Picard rank $0$.
\begin{example}\label{e:d8}
Let $\ovl{X}$ be the blow-up of $\PP^2$ at a point $P \in \PP^2(k)$. Let $C \subseteq \PP^2$ be a quadric passing through $P$ and let $L \subseteq \PP^2$ be a line which does not contain $P$ and meets $C$ at two distinct points $Q_0,Q_1 \in \PP^2(k)$, both defined over $k$. Let $\wtl{C}$ be the strict transform of $C$ in $\ovl{X}$. Then $D = \wtl{C} \cup L$ is a simple normal crossing divisor, and it is straightforward to check that $[D] = -K_{\ovl{X}}$. As we will see in Remark~\ref{r:simply-connected} below, the smooth variety $X = \ovl{X} \bksl D$ is simply connected, and so $X$ is an ample log K3 surface. Since $[L]$ and $[\wtl{C}]$ form a basis for $\Pic(\ovl{X}_{\ovl{k}})$, Lemma~\ref{l:basis} implies that $X$ has Picard rank $0$.

To construct explicit equations for $X$, let $x,y,z$ be projective coordinates on $\PP^2$ such that $L$ is given by $z=0$. Let $f(x,y,z)$ be a quadratic form vanishing on $C$ and let $g(x,y,z)$ be a linear form such that the line $g = 0$ passes through $P$ and $Q_1$. Then $X$ is isomorphic to the affine variety given by the equation
$$ f(x,y,1)t = g(x,y,1) .$$
By a linear change of coordinates we may assume that $Q_0 = (1,0,0)$ and $Q_1 = (0,1,0)$, in which case the equation above can be written as
$$ (axy + bx + cy + d)t = ex + f .$$
For some $a,b,c,d,e,f \in k$. We will later see that the $k$-isomorphism type of this surface does not depend, in fact, on any of these parameters.
\end{example}

\begin{example}\label{e:d7}
Let $L = k(\sqrt{a})$ be a quadratic extension of $k$. Let $P_0,P_1 \in \PP^2(L)$ a $\Gal(L/k)$-conjugate pair of points and let $\ovl{X}$ be the blow-up of $\PP^2$ at $P_0$ and $P_1$. Let $L \subseteq \PP^2$ be a line defined over $k$ which does not meet $\{P_0,P_1\}$ and let $L_1,L_2 \subseteq \PP^2$ be a $\Gal(L/k)$-conjugate pair of lines such that $L_1$ contains $P_1$ but not $P_2$ and $L_2$ contains $P_2$ but not $P_1$. Assume that the intersection point of $L_1$ and $L_2$ is not contained in $L$. Let $\wtl{L}_1,\wtl{L}_2$ be the strict transforms of $L_1$ and $L_2$ in $\ovl{X}$. Then $D = L \cup \wtl{L}_1 \cup \wtl{L}_2$ is a simple normal crossing divisor, and it is straightforward to check that $[D] = -K_{\ovl{X}}$. As we will see in Remark~\ref{r:simply-connected} below the smooth variety $X = \ovl{X} \bksl D$ is simply connected and so $X$ is an ample log K3 surface. Since $[L],[\wtl{L}_1]$ and $[\wtl{L}_2]$ form a basis for $\Pic(\ovl{X}_{\ovl{k}})$, Lemma~\ref{l:basis} implies that $X$ has Picard rank $0$.

To construct explicit equations for $X$, let $x,y,z$ be projective coordinates on $\PP^2$ such that $L$ is given by $z=0$. Let $f_1(x,y,z)$ and $f_2(x,y,z)$ be linear forms defined over $L$ which vanish on $L_1$ and $L_2$ respectively. Let $g(x,y,z)$ be a linear form defined over $k$ such that the line $g = 0$ passes through $P_1$ and $P_2$. Then $X$ is isomorphic to the affine variety given by the equation
$$ f_1(x,y,1)f_2(x,y,1)t = g(x,y,1) .$$
By a linear change of variables (over $k$) we may assume that the intersection point of $L_1$ and $L_2$ is $(0,0,1)$ and that $f(x,y) = x + \sqrt{a}y$, in which case the equation above becomes 
$$ (x^2 - ay^2)t = bx + cy + d .$$
We will later see that the $k$-isomorphism type of this surface only depends on the quadratic extension $L/k$, i.e., only on the class of $a$ mod squares.
\end{example}

\section{Geometry of ample log K3 surfaces with Picard rank $0$}
Let $k$ be a field of characteristic $0$. The goal of this section is to classify all ample log K3 surfaces $X/k$ of Picard rank is $0$.

\subsection{The category of log K3 structures}

\begin{define}
Let $X$ be a log K3 surface. We will denote by $\Log(X)$ the category whose objects are log K3 structures $(\ovl{X},D,\iota)$ and whose morphisms are maps of pairs $f:(\ovl{X},D) \lrar (\ovl{X}',D')$ which respect the embedding of $X$.
\end{define}

Between every two objects of $\Log(X)$ there is at most one morphism. We may hence think of $\Log(X)$ as a \textbf{partially ordered set} (poset) where the various log K3 structures are ordered by dominance. The goal of this section is to show that $\Log(X)$ is in fact a \textbf{cofiltered} poset, i.e., every two objects are dominated by a common third (see Corollary~\ref{c:cofiltered} below). In that sense the choice of a log K3 structure on a given log K3 surface is ``almost unique''. We begin with the following well-known statement (see the introduction section in~\cite{Fr15}), for which we include a short proof for the convenience of the reader.
\begin{prop}\label{p:cycle}
Let $\ovl{X}$ be a smooth, projective, (geometrically) rational surface over $k$ and $D \subseteq \ovl{X}$ a simple normal crossing divisor such that $[D]=-K_{\ovl{X}}$. Then one of the following option occurs:
\begin{enumerate}
\item
$D$ is a geometrically irreducible smooth curve of genus $1$.
\item
The geometric components of $D$ are all of genus $0$, and the dual graph of $D$ is a circle containing at least two vertices.
\end{enumerate}
\end{prop}
\begin{proof}
Since this is a geometric statement we may as well extend our scalars to $\ovl{k}$. According to~\cite[Lemma II.5]{Ha97} the underlying curve of $D$ is connected. Now let $D^0 \subseteq D$ be a geometric component and let $E \subseteq D$ be the union of the geometric components which are different from $D'$. Since $[D^0] + [E] = [D]$ is the anti-canonical class the adjunction formula tells us that
$$ 2 - 2g(D^0) = [D^0]\cdot ([D] - [D^0]) = [D^0]\cdot[E] $$
Since $D^0$ and $E$ are effective divisors without common components it follows that $[D^0]\cdot[E] \geq 0$ and hence either $g(D^0) = 1$ and $[D^0]\cdot[E] = 0$ or $g(D^0) = 0$ and $[D^0]\cdot[E] = 2$. Since $D$ is connected it follows that $D$ is either a genus $1$ curve or a cycle of genus $0$ curves. 
\end{proof}

Let us now establish some notation. Let $X$ be a log K3 surface of Picard rank $0$ equipped with a log K3 structure $(\ovl{X},D,\iota)$. We will denote by $d = [D]\cdot [D]$ and $n=\rank(\Pic(\ovl{X}_{\ovl{k}}))$. We note that since $\Pic(X_{\ovl{k}}) = 0$ the surface $X$ cannot be proper and so Proposition~\ref{p:rational} implies that $\ovl{X}_{\ovl{k}}$ is rational. It follows that $d$ and $n$ are related by the formula $n=10-d$. We will refer to $d$ as the \textbf{degree} of the log K3 structure $(\ovl{X},D,\iota)$. If $(\ovl{X},D,\iota)$ is ample then $d > 0$.

Since $X$ has Picard rank $0$, Lemma~\ref{l:basis} implies that the number of geometric components of $D$ is equal to $n$. According to Proposition~\ref{p:cycle}, we either have that $D$ is a smooth curve of genus $1$ and $d=9$ or $D$ is a cycle of $n\geq 2$ genus $0$ curves. The case where $D$ is a curve of genus $1$ and $d=9$ cannot occur, because then $\ovl{X} = \PP^2$ and no genus $1$ curve forms a basis for $\Pic(\PP^2)$. Hence $D$ is necessarily a cycle of $n$ genus $0$ curves. In particular, we have the following corollary:
\begin{cor}
Let $X$ be a log K3 surface of Picard rank $0$ equipped with an ample log K3 structure $(\ovl{X},D,\iota)$ of degree $d$. Then $1 \leq d \leq 8$ and $D$ is a cycle of $10-d$ genus $0$ curves.
\end{cor}

\begin{define}
Let $X$ be a log K3 surface equipped with a log K3 structure $(\ovl{X},D,\iota)$. If $D_1,...,D_n$ are the geometric components of $D$, numbered compatibly with the cyclic order of $G(D)$ (see Proposition~\ref{p:cycle}), then we will denote by $a_i =  [D_i]\cdot[D_i]$. Following the terminology of~\cite{Fr15} we will refer to $(a_1,...,a_n)$ as the \textbf{self-intersection sequence} of $(\ovl{X},D,\iota)$. We note that the self-intersection sequence is well-defined up to a dehidral permutation.
\end{define}

\begin{example}\label{e:d8-sig}
Let $X$ be a log K3 surface of the form described in Example~\ref{e:d8}, and let $(\ovl{X},D,\iota)$ be the associated log K3 structure. Then $\ovl{X}$ is a del Pezzo surface of degree $8$ and the components of $D$ consist of a rational curve $L$ with self intersection $1$ and a rational curve $\wtl{C}$ with self intersection $3$, both defined over $k$. Furthermore, by construction the intersection points of $\wtl{C}$ and $L$ are defined over $k$. It follows that the self-intersection sequence of $(\ovl{X},D,\iota)$ is $(3,1)$ and the Galois action on $G(D)$ is trivial.
\end{example}

\begin{example}\label{e:d7-sig}
Let $X$ be a log K3 surface of the form described in Example~\ref{e:d7}, and let $(\ovl{X},D,\iota)$ be the associated log K3 structure. Then $\ovl{X}$ is a del Pezzo surface of degree $7$ and the components of $D$ consist of a rational curve $L$ with self intersection $1$ and two $\Gal(k(\sqrt{a})/k)$-conjugate rational curves $\wtl{L}_1,\wtl{L}_2$, defined over $k(\sqrt{a})$, each with self intersection $0$. We hence see that the self-intersection sequence of $(\ovl{X},D,\iota)$ is $(0,0,1)$ and the Galois action on $G(D)$ factors through the quadratic extension $k(\sqrt{a})/k$, where the generator of $\Gal(k(\sqrt{a})/k)$ acts by switching the two $0$-curves. In particular, the splitting field of $D$ is $k(\sqrt{a})$.
\end{example}

We shall now consider a basic construction which allows one to change the log K3 structure of a given log K3 surface by performing blow-ups and blow-downs. We will see below that any two log K3 structures on a given log K3 surface can be related to each other by a sequence of such blow-ups and blow-downs.
\begin{const}\label{c:blow-up-down}
Let $X$ be a log K3 surface and let $(\ovl{X},D,\iota)$ be a log K3 structure on $X$ of degree $d$ and self-intersection sequence $(a_1,...,a_n)$. If the intersection point $P_i = D_i \cap D_{i+1}$ is defined over $k$, then we may blow-up $\ovl{X}$ at $P_i$ and obtain a new simple compactification $(\ovl{X}',D',\iota')$. It is then straightforward to verify that $[D'] = -K_{\ovl{X}'}$ and hence $(\ovl{X}',D',\iota')$ is a log K3 structure of degree $d-1$ and self-intersection sequence $(a_1,...,a_i-1,-1,a_{i+1}-1,a_{i+2},...,a_n)$. If $a_i,a_{i+1} \neq -1$ then a direct application of the adjunction formula shows that $P$ cannot lie on any $(-1)$-curve of $X$. If in addition $(\ovl{X},D,\iota)$ is ample and $d > 1$ then $(\ovl{X}',D',\iota')$ is ample by the Nakai--Moishezon criterion. Following the notation of~\cite{Fr15} we will refer to such blow-ups as \textbf{corner blow-ups}.

Alternatively, if $n \geq 3$ and for some $i=1,...,n$, the geometric component $D_i$ is a $(-1)$-curve defined over $k$, then we may blow-down $\ovl{X}$ and obtain a new simple compactification $(\ovl{X}',D',\iota')$. It is then straightforward to verify that $[D'] = -K_{\ovl{X}'}$ and hence $(\ovl{X}',D',\iota')$ is a log K3 structure of degree $d+1$ and self-intersection sequence $(a_1,...,a_{i-1}+1,a_{i+1}+1,a_{i+2},...,a_n)$. Furthermore, if $(\ovl{X},D,\iota)$ is ample then so is $(\ovl{X}',D',\iota')$. Following the notation of~\cite{Fr15} we will refer to such blow-downs as \textbf{corner blow-downs}.
\end{const}

\begin{rem}\label{r:trivial}
Let $X$ be a log K3 surface of Picard rank $0$ and let $(\ovl{X},D,\iota)$ be an ample log K3 structure on $X$. Let $\ovl{X}' \lrar \ovl{X}$ be a new log K3 structure obtained by a corner blow-up (see Construction~\ref{c:blow-up-down}). It is then clear that the Galois action on $G(D)$ is trivial if and only if the Galois action on $G(D')$ is trivial. In particular, corner blow-ups preserve the splitting field of the compactification.
\end{rem}

\begin{prop}\label{p:cofiltered}
Let $X$ be a log K3 surface and let $(\ovl{X},D,\iota)$ and $(\ovl{X}',D',\iota')$ be two log K3 structures. Then $(\ovl{X}',D',\iota')$ can be obtained from $(\ovl{X},D,\iota)$ by first performing a sequence of corner blow-ups and then performing a sequence of corner blow-downs.
\end{prop}
\begin{proof}
Clearly we may assume that $D$ and $D'$ are not empty. Since $k$ has characteristic $0$ we may find a third simple compactification $(Y,E,\iota_Y)$, equipped with compatible maps
$$ \xymatrix{
& Y \ar_{p}[dl]\ar^{q}[dr] & \\
\ovl{X} && \ovl{X}' \\
}$$
where both $p$ and $q$ can be factored as a sequence of blow-down maps (defined over $k$)
$$ (Y,E) = (\ovl{X}_m,D_m) \x{p_m}{\lrar} (\ovl{X}_{m-1},D_{m-1}) \x{p_{m-1}}{\lrar} \hdots \x{p_1}{\lrar} (\ovl{X}_0,D_0) = (\ovl{X},D) $$
and
$$ (Y,E) = (\ovl{X}'_k,D'_k) \x{q_k}{\lrar} (\ovl{X}'_{k-1},D'_{k-1}) \x{q_{k-1}}{\lrar} \hdots \x{q_1}{\lrar} (\ovl{X}'_0,D'_0) = (\ovl{X}',D') $$
such that the center of $p_i$ is contained in $D_{i-1}$ and the center of $q_j$ is contained in $D'_{j-1}$. Furthermore, we may choose $Y$ to be such that the $m+k$ attains its minimal possible value. Since $D$ and $D'$ are simple normal crossing divisors it follows that each $D_i$ and $D'_j$ are simple normal crossing divisors. We further note that every geometric fiber of either $p$ or $q$ is connected.

According to Proposition~\ref{p:cycle}, $D$ is either a geometrically irreducible smooth curve of genus $1$ or a cycle of genus $0$ curves. Let us first treat the case where $D$ is a curve of genus $1$. Let $\wtl{D}$ be the strict transform of $D$ in $Y$. Since $\wtl{D}$ has genus $1$ the image $q(\wtl{D})$ cannot be a point, and hence $q(\wtl{D}) = D'$. Since $D'$ is smooth and the fibers of $q$ are connected $q$ must induce an isomorphism $\wtl{D} \x{\cong}{\lrar} D'$. In this case the birational transformation $\xymatrix{\ovl{X} \ar@{-->}[r] & \ovl{X}'}$ has no fundamental points and hence extends to an well-defined map $\ovl{X} \lrar \ovl{X}'$. Arguing the same in the other direction we get that $\xymatrix{\ovl{X}' \ar@{-->}[r] & \ovl{X}}$ has no fundamental points as well and hence we get an isomorphism $\ovl{X} \cong \ovl{X}'$. In particular, the desired result holds vacuously.

Now assume that $D$ is a cycle of genus $0$ curves. By the above argument $D'$ must be a cycle of genus $0$ curves as well. In particular, $G(D)$ and $G(D')$ are circles. We shall now define for each $i=0,...,m$ a simple circle $C_i \subseteq G(D_i)$ as follows. For $i=0$ we set $C_0 = G(D_0)$. Now suppose that $C_i \subseteq G(D_i)$ has been defined for some $i \geq 0$. Recall that $\ovl{X}_{i+1}$ is obtained from $\ovl{X}_i$ by blowing up a point on $D_i$. If the blow-up point is not an intersection point corresponding to an edge of $C_i$ then the strict transforms of the components in $C_i$ form a simple circle in $G(D_{i+1})$ and we define $C_{i+1}$ to be the resulting circle. If the blow-up point is an intersection point corresponding to an edge of $C_i$ then the sequence of strict transforms of the components in $C_i$ forms a simple chain, and we may close this chain to a simple circle by adding the exceptional curve of $\ovl{X}_{i+1}$. We then define $C_{i+1}$ to be the resulting circle. Similarly, we may define a circle $C_j' \subseteq G(D'_j)$ for $j=0,...,k$. These constructions yield two simple circles in the dual graph $G(E) = G(D_m) = G(D'_k)$, which must coincide since $H^1(|G(E)|,\ZZ) \cong \ZZ$ (where $|\bullet|$ denotes geometric realization) and so $G(E)$ cannot contain two distinct simple circles. In particular, the geometric components of $D'_i$ which belong to $C_i'$ are exactly the geometric components which are images of geometric components lying in $C_m \subseteq G(D_m) = G(D'_k)$.

We now claim that the minimality of $(Y,E)$ implies that $C_i = G(D_i)$ for every $i$. Indeed, let $i_0 > 0$ be the smallest index such that $C_{i_0} \neq G(D_{i_0})$. This means that $\ovl{X}_{i_0}$ is obtained from $\ovl{X}_{i_0-1}$ by blowing up a point $P \in D_{i_0-1}$ which lies on exactly one geometric component $D^0_{i_0-1} \subseteq D_{i_0-1}$. We may then define inductively $D^0_i \subseteq D_i$ for $i \geq i_0$ to be the strict transform of $D^0_{i-1}$. This gives us in particular a component $E^0 = D^0_m$ in $E = D_m$. Let $j_0$ be the smallest index such that the image of $E^0$ in $D'_{j_0}$ is $1$-dimensional. We may now define for each $j \geq j_0$ the component $D^1_j \subseteq D'_j$ to be the image of $E^0$. In addition, we may define for each $i \geq i_0$ the curve $T_i \subseteq D_i$ which is the inverse image of $P$, and for each $j \geq j_0$ the curve $T'_j \subseteq D'_j$ which is image of $T_m \subseteq D_m = E$. We note that $T_i$ need not be irreducible, but must be connected. By construction, no geometric component of $T_i$ is a vertex of $C_i$ and no geometric component of $T'_i$ is a vertex of $C'_i$. We now claim that $T_{j_0}$ is a point. If $j_0 = 0$ then this follows from the fact that every component of $D'_0$ is a vertex of $C'_0$. If $j_0 > 0$ then by the choice of $j_0$ the map $q_{j_0}:\ovl{X}'_{j_0} \lrar \ovl{X}'_{j_0-1}$ must be the blow-down of $D^1_{j_0}$. Since $D'_{j_0-1}$ has simple normal crossings it follows that $D^1_{j_0}$ can have at most two intersection points with other components of $D'_{j_0}$. Since $D^1_{j_0}$ belongs to the circle $C'_{j_0}$ it already has two intersection points with components in $C'_{j_0}$, and hence these must be the only intersection points on $D^1_{j_0}$. This implies that $T'_{j_0}$ is a point (otherwise there would be a third intersection point of $D^1_{j_0}$ with some component of $T'_{j_0}$).

Let $r$ be the number of components of $T_m \subseteq D_m$. We now observe that for $i \geq i_0$ the only component of $D_i$ which meets a component of $T_i$ is $D^0_i$, and that for $j > j_0$ the only component of $D'_j$ which meets a component of $T'_j$ is $D^1_j$. We may hence rearrange the order of blow-ups and blow-downs so that the components of $T_m$ are added in steps $i=m-r+1,...,m$ and are blown-down in steps $j=k,...,k-r+1$. The minimality of $(Y,E)$ now implies that $r = 0$. It follows that $C_i = G(D_i)$ for every $i$ as desired, and hence $C'_j = G(D'_j)$ for every $j$ as well. In particular, $E$ is a circle of curves, $(Y,E)$ is obtained from $(\ovl{X},D)$ by a sequence of blowing up intersection points, and $(\ovl{X}',D')$ is obtained from $(Y,E)$ by a sequence of blowing down components.
\end{proof}

\begin{cor}\label{c:cofiltered}
Let $X$ be a log K3 surface. Then the category $\Log(X)$ of log K3 structures on $X$ is \textbf{cofitlered}. In particular, if $(\ovl{X},D,\iota)$ and $(\ovl{X}',D',\iota')$ are two log K3 structures then there exists a third log K3 structure $(Y,E,\eta)$ equipped with maps $p:(Y,E) \lrar (\ovl{X},D)$ and $q:(Y,E) \lrar (\ovl{X}',D')$ which respect the embedding of $X$. 
\end{cor}

\begin{cor}\label{c:splitting-inv}
Let $X$ be a log K3 surface. Then the splitting fields of all log K3 structures on $X$ are identical. In particular, the splitting field is an invariant of $X$ itself.
\end{cor}
\begin{proof}
Combine Proposition~\ref{p:cofiltered} and Remark~\ref{r:trivial}.
\end{proof}

\subsection{The characteristic class}

In this subsection we will focus our attention on ample log K3 surfaces of Picard rank $0$. We will show that every such surface admits an ample log K3 structure of degree $5$ and self-intersection sequence $(-1,-1,-1,-1,-1)$. Associated with such a log K3 structure is a natural \textbf{characteristic class} $\alp \in H^1(k,\DD_5)$. We will show that this class is independent of the choice of a log K3 structure, and is hence an invariant of $X$ itself. We begin with some preliminary lemmas.

\begin{lem}\label{l:invariant}
Let $X$ be a log K3 surface of Picard rank $0$ and let $(\ovl{X},D,\iota)$ be a log K3 structure on $X$ of degree $d$ and self-intersection sequence $(a_1,...,a_n)$. Then
$$ \sum_i a_i = 3d - 20 .$$
\end{lem}
\begin{proof}
Since the geometric components of $D$ form a cycle we have
$$ d = [D] \cdot [D] = \sum_{i=1}^{n} [D_i]\cdot[D_i] + 2\sum_{i=1}^{n-1}[D_i][D_{i+1}] + 2[D_n]\cdot[D_0] = \sum_i a_i + 2n $$
and so
$$ \sum_i a_i = d - 2n = 3d - 20 $$
as desired.
\end{proof}

\begin{cor}
Let $X$ be a smooth, simply connected surface and $(\ovl{X},D,\iota)$ an ample log K3 structure on $X$ of degree $d$. Then $d \geq 5$.
\end{cor}
\begin{proof}
Let $(a_1,...,a_n)$ be the self-intersection sequence of $(\ovl{X},D,\iota)$. Since $\ovl{X}$ is a del Pezzo surface we have that $a_i \geq -1$ for every $i=1,..,n$. By Lemma~\ref{l:invariant} we get that
$$ 3d - 20 = \sum_i a_i \geq -n = d - 10 $$
and so $d \geq 5$, as desired.
\end{proof}

\begin{prop}\label{p:d8}
Let $X$ be a log K3 surface of Picard rank $0$ and let $(\ovl{X},D,\iota)$ be an ample log K3 structure of degree $8$. Then $(\ovl{X},D,\iota)$ has self-intersection sequence $(3,1)$. 
\end{prop}
\begin{proof}
Let us write $D = D_1 \cup D_2$ with $D_1,D_2$ the geometric components. According to Lemma~\ref{l:basis} the classes $\{[D_1],[D_2]\}$ form a basis for $\Pic(\ovl{X}_{\ovl{k}})$. Let $(a_1,a_2)$ be the self-intersection sequence of $D$. By Lemma~\ref{l:invariant} we have $a_1 + a_2 = 4$. Since $\ovl{X}$ is a del Pezzo surface we have $a_i \geq -1$. Hence, up to switching $[D_1]$ and $[D_2]$, the only options for $(a_1,a_1)$ are $(5,-1),(4,0),(3,1)$ and $(2,2)$. By Proposition~\ref{p:cycle} we have $[D_1] \cdot [D_2]=2$. Hence the intersection matrix of $D_1$ and $D_2$ is given by
$$ M = \left(\begin{matrix} a_1 & 2 \\ 2 & a_2 \\ \end{matrix}\right) .$$ 
Since $D_1, D_2$ form basis the determinant of this matrix must be $\pm 1$. We hence see that the only option for the self-intersection sequence is $(a_1,a_2) = (3,1)$.
\end{proof}

We are now in a position to show that any ample log K3 surface of Picard rank $0$ admits a log K3 structure of of self-intersection sequence $(-1,-1,-1,-1,-1)$.
\begin{prop}\label{p:classification-1}
Let $X$ be an ample log K3 surface of Picard rank $0$. Then $X$ admits an ample log K3 structure of degree $5$ and self-intersection sequence $(-1,-1,-1,-1,-1)$. 
\end{prop}
\begin{proof}
Let $(\ovl{X},D,\iota)$ be an ample log K3 structure on $X$ of degree $d \geq 5$ and self-intersection sequence $(a_1,...,a_n)$. First assume that $d=5$. By Lemma~\ref{l:invariant} we have $a_1 + a_2 + a_3 + a_4 + a_5 = -5$. Since $\ovl{X}$ is a del Pezzo surface each $a_i \geq -1$ and so $a_i= -1$ for every $i=1,...,5$. 

Now assume that $d=6$. In this case $n=4$ and by Lemma~\ref{l:invariant} we have $a_1 + a_2 + a_3 + a_4 = -2$ and hence $D$ contains either two or three $(-1)$-curves. The possible self-intersection sequences, up to a dehidral permutation, are then $(-1,-1,0,0), (-1,0,-1,0)$ and $(-1,-1,-1,1)$. Let us show that the self-intersection sequence $(-1,0,-1,0)$ cannot occur. For this, one may base change to $\ovl{k}$. One may then blow-down the two $(-1)$-curves and obtain a new ample log K3 structure of degree $8$ and self-intersection sequence $(2,2)$. According to Proposition~\ref{p:d8} this is impossible. It follows that the self-intersection sequence of $(\ovl{X},D,\iota)$ is either $(-1,-1,0,0)$ or $(-1,-1,-1,1)$.

If the self-intersection sequence is $(-1,-1,0,0)$ then the intersection point $P$ of the two $0$-curves must be Galois invariant. We may then perform a corner blow-up at $P$ and obtain new ample log K3 structure of degree $5$ and self-intersection sequence $(-1,-1,-1,-1,-1)$. If the self-intersection sequence is $(-1,-1,-1,1)$ then the $(-1)$-curve that meets two other $(-1)$-curves must be defined over $k$. We may then blow it down to obtain a new ample log K3 structure of degree $7$ and self-intersection sequence $(0,0,1)$. In this log K3 structure the two intersection points which lie on the $1$-curve must be defined over $k$ as a pair. We may then perform a corner blow-up at these two points to obtain an ample log K3 structure of degree $5$ and self-intersection sequence $(-1,-1,-1,-1,-1)$. This covers the case $d=6$.

Let us now assume that $d=7$. In this case $n=3$ and by Lemma~\ref{l:invariant} we have $a_1 + a_2 + a_3 = 1$. The possible self-intersection sequences (up to permutation) are then $(-1,-1,3), (-1,0,2), (-1,1,1)$ and $(0,0,1)$. We now claim that the self-intersection sequences $(-1,-1,3)$ and $(-1,1,1)$ cannot occur. Working again over $\ovl{k}$, we may blow-down one of the $(-1)$-curves and get an ample log K3 structure of degree $8$ and self-intersection sequence $(0,4)$ in the first case and self-intersection sequence $(2,2)$ in the second. According to Proposition~\ref{p:d8}, this is impossible. We next observe that $(-1,0,2)$ has a trivial symmetry group and that the symmetry group of $(0,0,1)$ fixes the vertex between the two $0$-curves. In particular, in either case there must exist an intersection point which is Galois invariant and which does not lie on any $(-1)$-curve. Performing a corner blow-up at this point we obtain an ample log K3 structure of degree $6$ and we can proceed as above.

Let us now assume that $d=8$. By Proposition~\ref{p:d8} the self-intersection sequence must be $(1,3)$. Blowing up the two intersection points we get an ample log K3 structure of degree $6$, and we may proceed as above.
\end{proof}

\begin{rem}\label{r:splitting}
Let $X$ be a log K3 surface and let $(\ovl{X},D,\iota)$ be an ample log K3 structure with splitting field $L$ (see Definition~\ref{d:dual}). Combining Proposition~\ref{p:classification-1} and Corollary~\ref{c:splitting-inv} we may conclude that $X$ admits an ample log K3 structure of self-intersection sequence $(-1,-1,-1,-1,-1)$ and splitting field $L$.
\end{rem}

\begin{rem}
According to~\cite[Theorem 2.1]{Va13} every del Pezzo surface of degree $5$ is rational over its ground field. Proposition~\ref{p:classification-1} now implies that every ample log K3 surface of Picard rank $0$ has a point defined over $k$.
\end{rem}

According to Proposition~\ref{p:classification-1} every ample log K3 surface of Picard rank $0$ admits an ample log K3 structure $(\ovl{X},D,\iota)$ such that $\ovl{X}$ is a del Pezzo surface of degree $5$ and $D$ is a cycle of five $(-1)$-curves. Let 
$$ \DD_5 = \left<\sig,\tau | \sig\tau\sig^{-1} = \tau^{-1}, \tau^5=1\right> $$ 
be the dehidral group of order $10$. A choice of two \textbf{neighbouring} vertices $v_0,v_1 \in G(D)$ yields an isomorphism $T_{v_0,v_1}: \DD_5 \lrar \Aut(G(D))$ which sends $\tau$ to a the rotation of $G(D)$ that maps $v_0$ to $v_1$ and $\tau$ to the reflection of $G(D)$ which fixes $v_0$. Since $\Aut(G(D))$ acts transitively on the set of ordered pairs of neighbours in $G(D)$ it follows that the isomorphism $T_{v_0,v_1}$ is well-defined up to conjugation.

\begin{define}
Let $X$ be a smooth log K3 surface. Given an ample log K3 structure $(\ovl{X},D,\iota)$ of self-intersection sequence $(-1,-1,-1,-1,-1)$ we will denote by 
$$ \rho_{\ovl{X}}: \Gam_k \lrar \Aut(G(D)) $$ 
the action of the Galois group $\Gam_k$ on the dual graph $G(D)$ of $D$. Choosing two neighbouring vertices $v_0,v_1 \in G(D)$ we denote by 
$$ \rho_{\ovl{X}}^{v_0,v_1} \x{\df}{=} T^{-1}_{v_0,v_1} \circ \rho_{\ovl{X}}: \Gam_k \lrar \DD_5 $$
the corresponding composition. As explained above, $T_{v_0,v_1}$ is independent of the choice of $v_0,v_1$ up to conjugation. We will denote by
$$ \alp_{\ovl{X}} \x{\df}{=} [\rho_{\ovl{X}}^{v_0,v_1}] \in H^1(\Gam_k,\DD_5) $$ 
the corresponding non-abelian cohomology class (which is independent of $v_0,v_1$) and will refer to it as the \textbf{characteristic class} of the log K3 structure $\ovl{X}$.
\end{define}

Our next goal is to show that $\alp_{\ovl{X}}$ does not depend on the choice of a log K3 structure $(\ovl{X},D,\iota)$. We begin by observing that the geometric realization $|G(D)|$ is homeomorphic to the $1$-circle and hence the first homology group $H^1(|G(D)|,\ZZ)$ is isomorphic to $\ZZ$. The induced action of $\Aut(G(D))$ in $H^1(|G(D)|,\ZZ)$ yields a homomorphism $\chi:\Aut(G(D)) \lrar \{1,-1\}$. Elements of $\Aut(G(D))$ which are mapped to $1$ acts via orientation preserving maps, i.e., via rotation of the circle, while elements which are mapped to $-1$ act via orientation reversing maps, i.e., via reflections. We will generally refer to elements of the first kind as rotations and elements of the second kind as reflections.

Given an isomorphism $T_{v_0,v_1}: \DD_5 \x{\cong}{\lrar} \Aut(G(D))$ as above we may identify $\chi$ with the homomorphism (denoted by the same name) 
$$ \chi:\DD_5 \lrar \{1,-1\} $$ 
which sends $\tau$ to $1$ and $\sig$ to $-1$. We note that this identification does not depend on the choice of $(v_0,v_1)$. We will refer to $\chi$ as the \textbf{sign homomorphism}. The sign homomorphism $\chi: \DD_5 \lrar \{1,-1\}$ induces a map $\chi_*: H^1(\Gam_k,\DD_5) \lrar H^1(\Gam_k,\mu_2)$. Given a log K3 structure $(\ovl{X},D,\iota)$ of self-intersection sequence $(-1,-1,-1,-1,-1)$ we hence obtain a quadratic character $\chi_*(\alp_{\ovl{X}}) \in H^1(\Gam_k,\mu_2)$. As a first step towards the invariance of $\alp_{\ovl{X}}$ we will show that $\chi(\alp_{\ovl{X}})$ is independent of $\ovl{X}$.

\begin{prop}\label{p:character}
Let $X$ be an ample log K3 surface of Picard rank $0$ and let $(\ovl{X},D,\iota)$ be a log K3 structure of self-intersection sequence $(-1,-1,-1,-1,-1)$. Then for every large enough $l$, the second $l$-adic cohomology with compact support $H^2_c(X_{\ovl{k}},\QQ_l)$ is isomorphic to $\mathbb{Q}_l$ and the action of $\Gam_k$ on $\mathbb{Q}_l$ is given by the quadratic character $\chi(\alp_{\ovl{X}})$.
\end{prop}
\begin{proof}

Consider the canonical long exact sequence of cohomology with compact support:
$$ ... \lrar H^1_c\left(X_{\ovl{k}},\QQ_l\right) \lrar H^1\left(\ovl{X}_{\ovl{k}},\QQ_l\right) \lrar H^1\left(D_{\ovl{k}},\QQ_l\right) \lrar  $$ 
$$ \lrar H^2_c\left(X_{\ovl{k}},\QQ_l\right) \lrar H^2\left(\ovl{X}_{\ovl{k}},\QQ_l\right) \lrar H^2\left(D_{\ovl{k}},\QQ_l\right) \lrar  H^3_c\left(X_{\ovl{k}},\QQ_l\right) \lrar ... $$ 

Since $\Br(\ovl{X}_{\ovl{k}}) = 0$ and the components of $D$ form a basis for $\Pic(\ovl{X}_{\ovl{k}})$ it follows that the map $H^2\left(\ovl{X}_{\ovl{k}},\QQ_l\right) \lrar H^2\left(D_{\ovl{k}},\QQ_l\right)$ is an isomorphism. Since $\ovl{X}_{\ovl{k}}$ is simply connected we obtain an isomorphism of Galois modules
$$ H^1\left(D_{\ovl{k}},\QQ_l\right) \cong H^2_c\left(X_{\ovl{k}},\QQ_l\right) .$$

Let us first assume that the image of $\rho_{\ovl{X}}: \Gam_k \lrar \Aut(G(D))$ is either trivial or generated by a reflection. In this case the Galois action on $D$ must fix one of the intersection points $P \in D$. Since $D_{\ovl{k}}$ is a cycle of $5$ rational curves, and each rational curve is simply connected, we see that the category of finite \'etale coverings of $D$ can be identified with the category of finite coverings of the dual graph of $D$ (where we say that a map of graphs is covering if it induces a covering map after geometric realization). We may hence identify $\pi_1(\ovl{D}_{\ovl{k}},P)$ with the pro-finite completion of the fundamental group of $G(D)$, namely, with $\widehat{\ZZ}$. Furthermore, the Galois action on $\pi_1(\ovl{D}_{\ovl{k}},P)$ is given by the action of $\Gam_k$ on $\pi_1(G(D),P)$ via $\rho_{\ovl{X}}$. In particular, the action of $\Gam_k$ on $\pi_1(\ovl{D}_{\ovl{k}},P) \cong \widehat{\ZZ}$ is given by the character $\Gam_k \x{\rho_{\ovl{X}}}\lrar \Aut(G(D)) \x{\chi}{\lrar} \{1,-1\}$. We may hence identify $H^2_c\left(X_{\ovl{k}},\QQ_l\right) \cong H^1\left(D_{\ovl{k}},\QQ_l\right)$ with $\Hom(\widehat{\ZZ},\ZZ_l) \otimes \QQ_l \cong \QQ_l$ and we get that the Galois action is given by the quadratic character $[\chi_*\alp_{\ovl{X}}] \in H^1(\Gam_k,\mu_2)$, as desired.

Let us now assume that the image of $\rho_{\ovl{X}}$ is not trivial and not generated by a reflection. In this case the image of $\rho_{\ovl{X}}$ is either the rotation subgroup or all of $\Aut(G(D))$. It will suffice to prove the claim for a single large enough $l$. We may hence assume that $l-1$ is not divisible by $5$. Let $H \subseteq \Aut(G(D))$ be the image of $\rho_{\ovl{X}}$. If $H$ contains a reflection (i.e., if $H = \Aut(G(D))$) then let $H_0 \subseteq H$ be a subgroup generated by a reflection. If $H$ does not contain a reflection (i.e., $H$ is the rotation subgroup of $\Aut(G(D))$), we let $H_0 = \{1\} \subseteq H$ be the trivial subgroup. Let $L/k$ be the field extension of degree $[H:H_0] = 5$ corresponding to the subgroup $\rho_{\ovl{X}}^{-1}(H_0) \subseteq \Gam_k$
%Then the Galois closure of $L$ is the $H$-Galois extension determined by $\rho_{\ovl{X}}$. 
and let $\rho_L: \Gam_L \lrar H_0$ be the natural map. By the argument above we know that $\Gam_L$ acts on $H^2_c(X_{\ovl{k}},\QQ_l) \cong \QQ_l$ via the map 
$$ \Gam_L \x{\rho_L}{\lrar} H_0 \x{\chi|_{H}}{\lrar} \{1,-1\} \lrar \QQ_l^*.$$ 
It follows that $\Gam_k$ acts on $H^2_c(X_{\ovl{k}},\QQ_l)$ via $\rho_{\ovl{X}}$. We need to prove that the induced action of $H$ on $\QQ_l$ is via $\chi$. %Since the case where $H$ is either trivial or generated by a reflection was covered in the first part of the proof we may assume that $H$ is either $\Aut(G(D))$ or the cyclic subgroup of rotations. 
If $H = \Aut(G(D))$ then since $\Aut(\QQ_l) = \QQ_l^*$ is abelian and $\chi$ exhibits $\{1,-1\}$ as the abelianization of $\Aut(G(D))$ we get that any action of $H$ on $\QQ_l$ factors through $\chi$. It is hence left to show that the action of $H$ is non-trivial. But this follows from the fact that in this case $H_0 \subseteq H$ is the subgroup generated by a single reflection, and hence the restricted map $\chi|_{H_0}:H_0 \lrar \{1,-1\}$ is an isomorphism. If $H \subseteq \Aut(G(D))$ is the cyclic subgroup of rotations than we note that by our assumption $\QQ_l^*$ contains no $5$-torsion elements and hence every action of $H$ on $\QQ_l$ is trivial, and in particular given by the trivial homomorphism $\chi|_{H}$.
\end{proof}

Finally, we may now show that $\alp_{\ovl{X}}$ is independent of $\ovl{X}$.
\begin{thm}\label{t:well-defined}
Let $X$ be a log K3 surface and let $(\ovl{X},D,\iota)$ and $(\ovl{X}',D',\iota')$ be two ample log K3 structures of self-intersection sequence $(-1,-1,-1,-1,-1)$. Then $G(D)$ and $G(D')$ are Galois equivariantly isomorphic. In particular, $\alp_{\ovl{X}} = \alp_{\ovl{X}'}$ in $H^1(\Gam_k,\DD_5)$.
\end{thm}
\begin{proof}
By Proposition~\ref{p:cofiltered} we may find a third log K3 structure $(Y,E,\iota_Y)$, equipped with compatible maps
$$ \xymatrix{
& Y \ar_{p}[dl]\ar^{q}[dr] & \\
\ovl{X} && \ovl{X}' \\
}$$
where both $p$ and $q$ can be factored as a sequence of blow-down maps. In particular, every geometric fiber of either $p$ (resp. $q$) is connected, and hence the pre-image under $p$ (resp. $q$) of every connected subscheme is connected. Let $D_0,...,D_4$ be the geometric components of $D$ and let $D'_0,...,D'_4$ be the components of $D'$. Let $\wtl{D}_0,...,\wtl{D}_4$ be the strict transform of the $D_i$'s in $Y$ and let $\wtl{D}'_0,...,\wtl{D}'_4$ be the strict transform of the $D'_i$'s in $Y$.

Let $\Lam \subseteq \Gam_k$ be the kernel of $\rho_{\ovl{X}}$ and let $L/k$ be the finite Galois extension determined by $\Lam$. By Proposition~\ref{p:character} $\rho_{\ovl{X}'}(\Lam)$ is contained in the rotation subgroup of $\Aut(G(D'))$. Since each $\wtl{D}_i$ is defined over $L$ we get that $q(\wtl{D}_i)$ is either a component of $D'$, or a point of $D'$, which in either case are defined over $L$. Since any rotation in $\Aut(G(D'))$ acts freely on the set of points over any field, it follows that $\Lam$ is contained in the kernel of $\rho_{\ovl{X}'}$. Applying the argument in the other direction we may conclude that $\Lam$ is equal to the kernel of $\rho_{\ovl{X}'}$. In particular, all the $D'_i$ are defined over $L$.

Now if $\Lam = \Gam_k$ then clearly $G(D)$ and $G(D')$ are equivariantly isomorphic (and $\alp_{\ovl{X}} = \alp_{\ovl{X}'} = 0$). If $\rho_{\ovl{X}}(\Gam_k) \subseteq \Aut(G(D))$ is a group of order $2$ generated by a reflection then by the above the same is true for $\rho_{\ovl{X}'}(\Gam_k)$ and the desired isomorphism follows from Proposition~\ref{p:character}. We may hence assume that $\Gam_k$ acts transitively on the geometric components of both $D$ and $D'$. 

Since $X$ is not proper and $E$ cannot be a genus $1$ curve, we get from Proposition~\ref{p:rational} and Proposition~\ref{p:cycle} that $E$ is a cycle of genus $0$ curves. We will say that two components $E_0,E_1 \subseteq E$ are \textbf{$p$-neighbours} if they can be connected by a consecutive chain of components $E_0,F_1,F_2,..,F_n,E_1$ such that $p(F_i)$ is a point for every $i=1,...,n$. Similarly, we define the notion of $q$-neighbours. Now if $\wtl{D}_i$ and $\wtl{D}_j$ are $p$-neighbours then $D_i \cap D_j$ must be non-empty. Conversely, if $P \in D_i \cap D_j$ is an intersection point then $p^{-1}(P)$ is connected and must be a chain of curves in $E$ (on which $p$ is constant) which meets both $\wtl{D}_i$ and $\wtl{D}_j$. It follows that $\wtl{D}_i$ and $\wtl{D}_j$ are $p$-neighbours if and only if $D_i$ and $D_j$ are neighbours in $D$.  

We now consider two possible cases. The first case is when $\wtl{D}_0$ coincides with one of the $\wtl{D}_i'$. Since the Galois action is transitive on geometric components we then have that each $\wtl{D}_i$ coincides with one of the $\wtl{D}'_j$. This implies that two components $E_0,E_1 \subseteq E$ are $p$-neighbours if and only if they are $q$-neighbours. By the above we get that the associated $D_i \mapsto q(\wtl{D}_i)$ determines a Galois equivariant isomorphism $G(D) \cong G(D')$, as desired.

Now assume that $\wtl{D}_0$ does not coincides with one of the $\wtl{D}_i'$. In this case, Galois invariance implies that the sets $\{\wtl{D}_0,...,\wtl{D}_4\}$ and $\{\wtl{D}'_0,...,\wtl{D}'_4\}$ are disjoint. Since the Galois action is also transitive on intersection point we may deduce that for each intersection point $P$ of $D$ the curve $p^{-1}(P)$ must contain $\wtl{D}'_j$ for exactly one $j$. Since $p^{-1}(P)$ is connected it follows that $q(p^{-1}(P)) = D'_j$. Now let $\widehat{G}(D)$ denote the graph whose vertices are the intersection points of $D$ and whose edges are the components of $D$. It follows that the association $P \mapsto q(p^{-1}(P))$ determines a Galois equivariant isomorphism of graphs $\widehat{G}(D) \cong G(D')$. Since $\widehat{G}(D)$ and $G(D)$ are equivariantly isomorphic (e.g., by sending each component to the antipodal intersection point) we may conclude that $G(D)$ and $G(D')$ are equivariantly isomorphic.

\end{proof}

\subsection{The classification theorem}

Relying on Proposition~\ref{p:classification-1} and Theorem~\ref{t:well-defined} we may now make the following definition:
\begin{define}
Let $X$ be an ample log K3 surface of Picard rank $0$. We will denote by $\alp_X \in H^1(k,\DD_5)$ the characteristic class $\alp_{\ovl{X}}$ associates to any log K3 structure $(\ovl{X},D,\iota)$ of self-intersection sequence $(-1,-1,-1,-1,-1)$.  
\end{define}

Our goal in this subsection is to prove the following theorem:
\begin{thm}\label{t:main-class-2}
The association $X \mapsto \alp_X$ determines a bijection between the set of $k$-isomorphism classes of ample log K3 surfaces of Picard rank $0$ and the cohomology set $H^1(k,\DD_5)$.
\end{thm}

The proof of Theorem~\ref{t:main-class-2} will occupy the rest of the subsection. We begin by a combinatorial result concerning the configuration of $(-1)$-curves on del Pezzo surfaces of degree $5$.

\begin{prop}\label{p:d5}
Let $\ovl{X}$ be a del Pezzo surface of degree $5$. Then there exist exactly $12$ cycles of five $(-1)$-curves on $\ovl{X}_{\ovl{k}}$, and the automorphism group of $\ovl{X}$ acts transitively on the set of such cycles. Furthermore, if $D_0,...,D_4$ is such a cycle then the $(-1)$-curves which are not in $\{D_0,...,D_4\}$ form a cycle $D_0',..., D_4'$ of length $5$ such that $D_i$ meets $D_j'$ if and only if $j = 2i$ mod $5$.
\end{prop}
\begin{proof}
It is well known that $\ovl{X}$ has exactly ten $(-1)$-curves and that two $(-1)$-curves are either skew or meet transversely at a single point. The intersection graph of the $(-1)$-curves is the \textbf{Petersen graph}:

\begin{figure}[h]\label{f:petersen}
\centering
\includegraphics[scale=0.3]{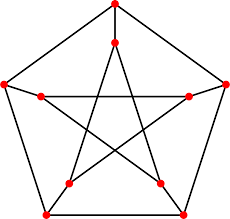}\caption[petersen]{The Petersen graph}
\end{figure}

Let $C_0,C_1$ be the external pentagon and internal pentagons in the Petersen graph. Then $C_0$ and $C_1$ each form a cycle of length five. Furthermore, each vertex of $C_0$ is connected by an edge to exactly one vertex of $C_1$, and vice versa. Now consider a cycle $C = \left<v_0,...,v_4\right>$ of length five which is not $C_0$ or $C_1$. Without loss of generality we may assume that $v_0 \in C_0$ and $v_4 \in C_1$. Since two edges connecting $C_0$ and $C_1$ cannot have a common vertex it follows that the intersection of $C$ with $C_0$ either consists of the segment $\left<v_0,v_1\right>$ or the segment $\left<v_0,v_1,v_2\right>$. We claim that for each choice of three consecutive vertices $w_0,w_1,w_2 \in C_0$ there exists a unique cycle of length five $C = \left<v_0,...,v_4\right>$ such that $v_i = w_i$ for $i=0,1,2$. Indeed, there could be at most one such cycle because $v_3, v_4 \in C_1$ are completely determined by the fact that $v_3$ is connected to $w_2$ and $v_4$ is connected to $w_0$. On the other hand, direct examination shows that this choice indeed gives a cycle, i.e., that $v_3$ and $v_4$ are connected in $C_1$. It follows that there are exactly five cycles of length five whose intersection with $C_0$ contains three vertices and by applying the above argument to $C_1$ instead of $C_0$ we see that there are also exactly five cycles of length five whose intersection with $C_0$ contains two vertices. We have hence counted all together (including $C_0$ and $C_1$ themselves) exactly $12$ cycles of length five. 

Now it is known that the automorphism group $G = \Aut(\ovl{X}_{\ovl{k}})$ of $\ovl{X}_{\ovl{k}}$ is isomorphic to $S_5$. Furthermore, the action of $G$ on the graph of $(-1)$-curves identifies $G$ with the isomorphism group of the Petersen graph. Let $H \subseteq G$ be the subgroup of those transformations mapping $C_0$ to itself. The automorphism group of a cycle of length five is $\DD_5$ and direct observation shows that each transformation in $\DD_5$ can be extended to an automorphism of the Petersen graph, and hence to an automorphism of $\ovl{X}_{\ovl{k}}$. It follows that $H \cong \DD_5$ and hence $H$ has exactly $12$ cosets in $G$, which means that $G$ acts transitively on the set of cycles of length five. 

It will now suffice to show the desired property for just one cycle of length five. In particular, it is straightforward to verify the property for the cycle $C_0$, with the remaining vertices forming the cycle $C_1$.
\end{proof}

We are now in a position to establish the first half of Theorem~\ref{t:main-class-2}, namely, that the invariant $\alp_X \in H^1(\Gam_k,\DD_5)$ determines $X$ up to a $k$-isomorphism.

\begin{thm}\label{t:uniqueness}
Let $X,Y$ be an ample log K3 surfaces of Picard rank $0$. If $\alp_X = \alp_Y$ then $X$ is $k$-isomorphic to $Y$. 
\end{thm}
\begin{proof}
Let $(\ovl{X},D,\iota)$ and $(\ovl{Y},E,\iota)$ be ample log K3 structures of degree $5$ and self-intersection sequence $(-1,...,-1)$ on $X$ and $Y$ respectively. Let $D_0,...,D_4$ be the geometric components of $D$, considered as a cycle of $(-1)$-curves. By Proposition~\ref{p:d5} the $(-1)$-curves of $\ovl{X}$ which are not in $\{D_0,...,D_4\}$ form a cycle $D_0',...,D_4'$ of length five, such that $D_i$ meets $D_j'$ if and only if $j=2i$ mod $5$. It follows that the association $i \mapsto 2j$ identifies the Galois action on the cycle $D_0',...,D_4'$ with the Galois action on the cycle $D_0,...,D_4$. We may hence extend the equivariant isomorphism of the dual graphs of $D$ and $E$ to an equivariant isomorphism of the graphs of $(-1)$-curves of $\ovl{X}$ and $\ovl{Y}$. By the general theory of del Pezzo surfaces of degree $5$, every such isomorphism is induced by an isomorphism $\ovl{X} \x{\cong}{\lrar} \ovl{Y}$, yielding the desired isomorphism $X \x{\cong}{\lrar} Y$. 
\end{proof}

%
%\subsection{Existence}\label{ss:explicit}

Our next goal is to show that any class $\alp \in H^1(k,\DD_5)$ is the characteristic class of some ample log K3 surface of Picard rank $0$. We begin with a few lemmas that will help us determine when certain open varieties are simply-connected.

\begin{lem}\label{l:pushout-group}
Let 
$$ \xymatrix{
K \ar^{f}[r]\ar^{g}[d] & H \ar[d] \\
G \ar[r] & P \\
}$$
be a pushout square of groups. If the map $f$ is surjective and the map $g$ vanishes on the kernel of $f$ then the map $G \lrar P$ is an isomorphism.
\end{lem}
\begin{proof}
By the universal property of pushouts it will suffice to show that for every homomorphism of groups $\vphi:G \lrar N$ there exists a unique homomorphism of groups $\psi: H \lrar N$ such that $\psi \circ f = \vphi \circ g$. But this now follows directly from the fact $\vphi \circ g$ vanishes on the kernel of $f$ and that $f$ is surjective.
\end{proof}

\begin{lem}\label{l:simply-connected}
Let $Y$ be a smooth, connected algebraic variety (not necessarily proper) over an algebraically closed field $\ovl{k}$ of characteristic $0$ and let $D \subseteq Y$ be a smooth irreducible divisor. Assume that $Y$ contains a smooth projective rational curve $C \subseteq Y$ which meets $D$ transversely at exactly one point. Let $W = Y \bksl D$. Then the induced map $\pi^{\et}_1(W,Q) \lrar \pi^{\et}_1(Y,Q)$ is an isomorphism.
\end{lem}
\begin{proof}
By choosing a large enough algebraically closed field which contains both $\ovl{k}$ and the field of complex numbers $\CC$, we may reduce to the case where $\ovl{k} = \CC$. Furthermore, by the comparison theorem we may replace the \'etale fundamental group by the topological fundamental group of the associated space of complex points.

Let $P$ be the intersection point of $C$ and $D$. Since $C$ meets $D$ transversely we may find a tobular neighborhood $D(\CC) \subseteq U \subseteq Y(\CC)$ such that $U \cap C(\CC)$ is a tobular neightborhood of $P$ in $C(\CC)$. Furthermore, we may find a retraction $r: U \lrar D(\CC)$ such that $U \cap C(\CC) = r^{-1}(P)$ and such that $r$ exhibits $U$ as a disc bundle over $D(\CC)$ (isomorphic to the disc bundle associated to the normal bundle of $D(\CC)$ in $Y(\CC)$). Let $E = U \cap W(\CC)$ and
let $q: E \lrar D(\CC)$ be the restriction of $r$ to $E$, so that $q$ is fiber bundle whose fibers are punctured discs (and in particular homotopy equivalent to $S^1$). Let $F = q^{-1}(P)$ and $Q \in F$ be a point. Using Van-Kampen's theorem and the fact that $r: U \lrar D(\CC)$ is a homotopy equivalence we obtain a pushout square of groups
$$ \xymatrix{
\pi_1(E,Q) \ar^{q_*}[r]\ar^{i_*}[d] & \pi_1(D(\CC),Q) \ar[d] \\
\pi_1(W(\CC),Q) \ar[r] & \pi_1(Y(\CC),Q) \\
}$$
Since $q: E \lrar D(\CC)$ is a fiber bundle whose fiber $F = q^{-1}(P)$ is connected it follows that the sequence of groups
$$ \pi_1(F,Q) \lrar \pi_1(E,Q) \lrar \pi_1(D) \lrar 1 $$
is exact. Since $F$ is contained in the contractible space $C(\CC) \cap W(\CC) \cong \A^1(\CC)$ (recall that $C \cong \PP^1$ by assumption) it follows that the composed map $\pi_1(F,Q) \lrar \pi_1(W(\CC),Q)$ is the trivial map. By Lemma~\ref{l:pushout-group} it follows that the map
$$ \pi_1(W(\CC),Q) \lrar \pi_1(Y(\CC),Q) $$
is an isomorphism as desired.
\end{proof}

\begin{cor}\label{c:simply-connected}
Let $\ovl{X}$ be a simply connected algebraic variety over an algebraically closed field $\ovl{k}$ of characteristic $0$. Let $D \subseteq \ovl{X}$ be a simple normal crossing divisor with geometric components $D_0,...,D_{n-1}$. Assume that for every $i=0,...,n-1$ there exists a smooth rational curve $E_i$ such that $D_i$ meets $E_i$ transversely at one point and $D_i \cap E_j = \emptyset$ if $i < j$. Then $X = \ovl{X} \bksl D$ is simply connected. 
\end{cor}
\begin{proof}
For each $r=0,...,n$ let $X_r = \ovl{X} \bksl \left[\cup_{i < r} D_i\right]$, so that in particular $X_0 = \ovl{X}$ and $X_n = X$. Let $Q \in X_n$ be a closed point. For each $r=0,...,n-1$ we may apply Lemma~\ref{l:simply-connected} with $Y = X_r$, $D = D_r$ and $C = E_r$ and deduce that the map $\pi^{\et}_1(X_{r+1},Q) \lrar \pi^{\et}_1(X_{r},Q)$ is an isomorphism. Since $\ovl{X}$ is simply connected it follows by induction that each $X_r$ is simply connected, and in particular $X = X_n$ is simply connected, as desired. 
\end{proof}

\begin{rem}\label{r:simply-connected}
Corollary~\ref{c:simply-connected} can be used to show that the surfaces considered in Examples~\ref{e:d8} and ~\ref{e:d7} are (geometrically) simply connected. In Example~\ref{e:d8} we may apply Corollary~\ref{c:simply-connected} with $E_0$ the inverse image of any line in $\PP^2$ other than $L$ and $E_1$ the exception curve above $P$. In Example~\ref{e:d7} we may apply Corollary~\ref{c:simply-connected} with $E_0$ the inverse image of any line in $\PP^2$ other than $L$ and $E_1,E_2$ the exceptional curves above $P_1$ and $P_2$ respectively.
\end{rem}

\begin{cor}
Let $\ovl{X}$ be a del Pezzo surface of degree $5$ and let $D \subseteq \ovl{X}$ be a cycle of five $(-1)$-curves. Then $\ovl{X} \bksl D$ is simply connected.
\end{cor}
\begin{proof}
We first note that any two $(-1)$-curves on $\ovl{X}$ which meet each other do so transversely. The desired result now follows by combining Proposition~\ref{p:d5} and Corollary~\ref{c:simply-connected}.
\end{proof}

\begin{lem}\label{l:basis-2}
Let $\ovl{X}$ be a del Pezzo surface of degree $5$. If $D_0,...,D_4$ is a cycle of $(-1)$-curves then $[D_0],...,[D_4]$ forms a basis for $\Pic(\ovl{X}_{\ovl{k}})$ and $\sum_i [D_i] = -K_{\ovl{X}}$.
\end{lem}
\begin{proof}
By Proposition~\ref{p:d5} it will suffice to prove this for a single choice of cycle. Since this is a geometric property we may as well extend our scalars to $\ovl{k}$ and identify $\ovl{X}$ with the blow-up of $\PP^2$ and four suitably generic points $P_1,P_2,P_3,P_4 \in \PP^2(\ovl{k})$. For $i=1,2,3$ let $L_i \subseteq \ovl{X}_{\ovl{k}}$ be the strict transform of the line in $\PP^2$ passing through $P_i$ and $P_{i+1}$ and for $j=1,...,4$ let $E_j$ be the exceptional curve associated to $P_j$. Then we have a cycle of $(-1)$-curves given by $\left<L_1,E_2, L_2, E_3, L_3\right>$. It is then straightforward to verify that these curves form a basis for $\Pic(\ovl{X}_{\ovl{k}})$, and that $\sum_i [D_i] = -K_{\ovl{X}}$.
\end{proof}

We now have what we need in order to establish the second half of Theorem~\ref{t:main-class-2}.
\begin{thm}\label{t:existence}
Let $k$ be a field of characteristic $0$ and let $\rho: \Gam_k \lrar \DD_5$ be a homomorphism. Then there exists a log K3 surface $X$ equipped with an ample log K3 structure $(\ovl{X},D,\iota)$ of self-intersection sequence $(-1,-1,-1,-1,-1)$ such that $\rho_{\ovl{k}} = \rho$.
\end{thm}
\begin{proof}
Let $\ovl{X}$ be the blow-up of $\PP^2$ at $4$ points in general position which are all defined over $k$. Then $\ovl{X}$ is a del Pezzo surface of degree $5$. Let $G$ be the incidence graph of $(-1)$-curves on $X$ (which is isomorphic to the Petersen graph, see~\ref{f:petersen}). By our construction the Galois action on $G$ is trivial. It is well-known that the action of $\Aut(\ovl{X}_{\ovl{k}})$ on $G$ induces an isomorphism $\Aut(\ovl{X}_{\ovl{k}}) \cong \Aut(G)$. Let us now choose a cycle $C$ of length five in $G$, and let $H \subseteq G$ be the stabilzer of this cycle. Then $H$ is isomorphic to $\DD_5$. Composing $\rho: \Gam_k \lrar \DD_5$ with the inclusion $\DD_5 \cong H \subseteq G$ we obtain a homomorphism $\sig:\Gam_k \lrar \Aut(G) \cong \Aut(\ovl{X}_{\ovl{k}})$. Twisting $\ovl{X}$ by $\sig$ we obtain a new del Pezzo surface $X^{\sig}$ over $k$. Since $\sig$ factors through the stablizer of $C$ we see that $C$ will be Galois stable in $\ovl{X}^{\sig}$. Let $D$ be the union of $(-1)$-curves in $C$. It now follows from Proposition~\ref{p:d5}, Lemma~\ref{l:basis}, Lemma~\ref{l:basis-2} and Corollary~\ref{c:simply-connected} that $X = \ovl{X} \bksl D$ is a log K3 surface. Furthermore, by construction we have $\rho_{\ovl{X}} = \rho$.
\end{proof}

We are now in a position to deduce our main result of this subsection.
\begin{proof}[Proof of Theorem~\ref{t:main-class-2}]
By Theorem~\ref{t:existence} every element $\alp \in H^1(k,\DD_5)$ is the characteristic class $\alp_X$ of some ample log K3 surface of Picard rank $0$. By Theorem~\ref{t:uniqueness} the invariant $\alp_X$ determines $X$ up to an isomorphism. It follows that the association $X \mapsto \alp_X$ determines a bijection between the set of $k$-isomorphism types of ample log K3 surfaces of Picard rank $0$ and the Galois cohomology set $H^1(k,\DD_5)$.
\end{proof}

\subsection{Quadratic log K3 surfaces}

\begin{define}
Let $X$ be an ample log K3 surface of Picard rank $0$. We will say that $X$ is \textbf{quadratic} with character $\chi \in H^1(X,\mu_2)$ if there exists a homomorphism $f:\{1,-1\} \lrar \DD_5$ such that $\alp_X = f_*(\chi)$.
\end{define}

\begin{rem}\label{r:quadratic}
By Remark~\ref{r:splitting} we see that a log K3 surface $X$ of Picard rank $0$ is quadratic if and only if $X$ admits an ample log K3 structure with a quadratic splitting field $L/k$, in which case $\chi$ is the quadratic character associated to $L$.
\end{rem}

\begin{define}
Given a non-zero element $a \in k^*$ we will denote by $[a] \in H^1(k,\mu_2)$ the class corresponding to the quadratic extension $k(\sqrt{a})/k$.
\end{define}

\begin{prop}\label{p:realization-2}
Let $X$ be a log K3 surface of Picard rank $0$ which is quadratic with character $\chi = [a] \in H^1(k,\{1,-1\})$. Then $X$ is isomorphic over $k$ to the affine surface in $\A^3$ given by 
\begin{equation}\label{e:quadratic} 
(x^2 - ay^2)t = y - 1
\end{equation}
\end{prop}
\begin{proof}
Equation~\ref{e:quadratic} is a particular case of Examples~\ref{e:d7} and~\ref{e:d7-sig} and hence defines a log K3 surface which possess an ample log K3 structure $(\ovl{X},D,\iota)$ of self-intersection sequence $(0,0,1)$ and splitting field $k(\sqrt{a})$. By Remark~\ref{r:quadratic} we see that $X$ is quadratic with character $\chi = [a]$. Finally, by Theorem~\ref{t:uniqueness} it follows that every quadratic log K3 structure with character $[a]$ is isomorphic to $X$.
\end{proof}

\begin{prop}\label{p:realization}
Let $X$ be an ample log K3 surface of Picard rank $0$. If $\alp_X = 0$ then $X$ is $k$-isomorphic to the affine surface in $\A^3$ given by the equation 
\begin{equation}\label{e:explicit}
(xy - 1)t = x - 1
\end{equation}
\end{prop}
\begin{proof}

Equation~\ref{e:quadratic} is a particular case of Examples~\ref{e:d8} and~\ref{e:d8-sig} and hence defines a log K3 surface which possess an ample log K3 structure $(\ovl{X},D,\iota)$ of self-intersection sequence $(3,1)$ and splitting field $k$. By Remark~\ref{r:quadratic} we see that $X$ is quadratic with character $\chi = 0$, and hence $\alp_X = 0$. By Theorem~\ref{t:uniqueness} it follows that every log K3 structure with $\alp_X = 0$ is isomorphic to $X$. 
\end{proof}

\section{Integral points}
\subsection{Zariski density}
Our goal in this section is to prove the following theorem:
\begin{thm}\label{t:zariski}
Let $\X$ be a separated smooth scheme of finite type over $\ZZ$ such that $X = \X \otimes_{\ZZ} \QQ$ is an ample log K3 surface of Picard rank $0$ and trivial characteristic class. Then the set of integral points $\X(\ZZ)$ is not Zariski dense.
\end{thm}
\begin{proof}
According to Proposition~\ref{p:realization} $X$ is isomorphic over $k$ to the affine surface in $\A^3$ given by the equation
\begin{equation}\label{e:K3}
(xy - 1)t = x - 1
\end{equation}
Hence the coordinates $x,y,t$ determine three rational functions $f_x,f_y,f_t$ on the scheme $\X$ which are regular when restricted to $X$. It follows that the poles of $f_x,f_y$ and $f_t$ are all ``vertical'', i.e., they are divisors of the form $M=0$ for $M \in \ZZ$. In particular, there exists divisible enough $M$ such that for every $P \in \X(\ZZ)$ the values $Mf_x(P), Mf_y(P)$ and $Mf_t(P)$ are all integers. Given an $M \in \ZZ$, we will say that a number $x \in \QQ$ is $M$-integral if $Mx \in \ZZ$. We will say that a solution $(x,y,t)$ of ~\ref{e:K3} is $M$-integral if each of $x,y$ and $t$ is $M$-integral. Now by the above there exists an $M$ such that $(f_x(P),f_y(P),f_z(P))$ is an $M$-integral solution of~\ref{e:K3} for every $P \in \X(\ZZ)$. It will hence suffice to show that for every $M \in \ZZ$, the set of $M$-integral solutions of~\ref{e:K3} is not Zariski dense (in the affine variety~\ref{e:K3}).

Since the function $f(y) = \frac{y-1}{y}$ on $\RR$ converges to $1$ as $y$ goes to either $\pm\infty$ it follows that there exists a positive constant $C > 0$ such that $\left|\frac{y}{y + 1}\right| < C$ for every $M$-integral number $-1 \neq y \in \QQ$. We now claim that if $(x,y,t)$ is a $M$-integral solution then either $|y| \leq 2M$ or $|x-1| \leq 2C$ or $t = 0$. Indeed, suppose that $(x,y,t)$ is an $M$-integral triple such that $|y| > 2M$, $|x-1| > 2C$, and $t \neq 0$. Then $|x-1| > 2\left|\frac{y-1}{y}\right|$ and hence 
$$ |(xy-1)t| = |(x-1)y + (y-1)||t| > \frac{1}{2M}|(x-1)y| > |x-1|, $$
which means that $(x,y,t)$ is not a solution to~\ref{e:K3}. It follows that all the $Q$-integral solutions of~\ref{e:K3} lie on either the curve $t = 0$, or on the curve $x-1 = i$ for $|i| \leq 2C$ an $M$-integral number, or the curve $y = j$ for $|j| \leq 2M$ an $M$-integral number. Since this collection of curves is finite it follows that $M$-integral solutions to~\ref{e:K3} are not Zariski dense.
\end{proof}

Theorem~\ref{t:zariski} raises the following question:
\begin{question}\label{q:zariski}
Is it true that $\X(\ZZ)$ is not Zariski dense for any integral model of any log K3 surface of Picard rank $0$?
\end{question}

We shall now show that the answer to question~\ref{q:zariski} is negative. Theorem I of ~\cite{Na88} implies, in particular, that there exists a real quadratic number field $L = \QQ(\sqrt{a})$, ramified at $2$ and with trivial class group, such that the reduction map $\OO_L^* \lrar (\OO_L/\mathfrak{p})^*$ is surjective for infinitely many prime ideals $\mathfrak{p} \subseteq \OO_L$ of degree $1$ over $\QQ$. Given such an $L$, we may find a square-free positive integer $a \in \QQ$ such that $L = \QQ(\sqrt{a})$. Now let $\X$ be the ample log K3 surface over $\ZZ$ given by the equation
\begin{equation}\label{e:d7-dense}
(x^2 - ay^2)t = y - 1.
\end{equation}
We now claim the following:
\begin{prop}
The set $\X(\ZZ)$ of integral points is Zariski dense. 
\end{prop}
\begin{proof}
Let $\mathfrak{p} = (\pi) \subseteq \OO_L$ be an odd prime ideal of degree $1$ such that $\OO_L^* \lrar (\OO_L/\mathfrak{p})^*$ is surjective and let $p = N_{L/\QQ}(\pi)$. We note that $L$ is necessarily unramified at $p$. Let $r \in \FF_p^*$ be the image of the residue class of $\sqrt{a}$ under the (unique) isomorphism $\OO_L/\mathfrak{p} \cong \FF_p$. Let $\sig \in \Gal(L/\QQ)$ be a generator. Then the image of the residue class of $\sqrt{a}$ under the isomorphism $\OO_L/\sig(\mathfrak{p}) \cong \FF_p$ is necessarily $-r$.

By our assumption on $L$ there exists a $u \in \OO_L^*$ such that the residue class of $u\sig(\pi)$ mod $\mathfrak{p}$ is equal to $2r$. Since $L$ is ramified at $2$ there exists $x_0,y_0 \in \ZZ$ such that $u\sig(\pi) = x_0 + \sqrt{a}y_0$. Let $\ovl{x}_0,\ovl{y}_0 \in \FF_p$ be the reductions of $x_0$ and $y_0$ mod $p$ respectively. By construction we have $\ovl{x}_0 - r\ovl{y}_0 = 0$ and $\ovl{x}_0 + r\ovl{y}_0 = 2r$ and hence $\ovl{x}_0 = r$ and $\ovl{y}_0 = 1$. It follows that $y_0 - 1$ is divisible by $p$ and since $x_0^2 - ay_0^2 = N_L(u\sig(\pi)) = \pm p$ there exists a $t_0 \in \ZZ$ such that 
$$ (x_0^2 - ay_0^2)t_0 = y_0 - 1 .$$
In particular, the triple $(x_0,y_0,t_0)$ is a solution for~\ref{e:d7-dense}. Let $\C_p \subseteq \X$ be the curve given by the additional equation $x^2 - ay^2 = p$. We have thus found an integral point on $\C_p$. By multiplying $u$ with units whose image in $\OO_L/\mathfrak{p}$ is trivial we may produce in this way infinitly many integral points on $\C_p$. Now any irreducible curve in $X$ is either equal to $C_p = \C_p \otimes_\ZZ \QQ$ for some $p$ or intersects each $C_p$ at finitely many points. Our construction above produces infinitely many $p$'s for which $\C_p$ has infinitely many integral points, and hence $\X(\ZZ)$ is Zariski dense, as desired.
\end{proof}

We end this section with the following question:
\begin{question}
Does conjecture~\ref{c:pic-0} hold for the surface~\ref{e:d7-dense}? If so, what is the appropriate value of $b$?
\end{question}
%
%\begin{conj}
%Let $k$ be a number field, $S \subseteq \Om_k$ a finite set of places of $k$, $L = k(\sqrt{a})$ a quadratic extension and $T \subseteq \Om_L$ the finite set of places of $L$ lying above $S$. Let $\X$ be a smooth surface over $\OO_S$ such that $\X \otimes_{\OO_S} k$ is an ample log K3 surface of Picard rank $0$ which is quadratic with character $[a]$. Let $U = \{u \in \OO_T^* | N_{L/k}(u) = 1\}$ be the group of $T$-units of norm $1$. Then $\X(\OO_S)$ is Zariski dense if and only if $U$ is infinite.
%\end{conj}
% 

\subsection{A counter-example to the Brauer-Manin obstruction}\label{ss:BM}

Let $k = \QQ$ and $S = \{\infty\}$. In this subsection we will prove Theorem~\ref{t:counter}, by constructing a log K3 surface over $\ZZ$ which is $S$-split and for which the integral Brauer-Manin obstruction is insufficient to explain the lack of integral points.

Let $a,b,c,d,m \in \ZZ$ and set $\Del = ad-bc$. Assume that $acm\Del \neq 0$. Consider the smooth affine scheme $\X \subseteq \A^3_{\ZZ}$ given by the equation
\begin{equation}\label{e:k3-2}
((ax+b)y + m)t = cx+d .
\end{equation}
Then $X = \X \otimes_{\ZZ} \QQ$ is a particular case of Example~\ref{e:d7} and and is hence an ample log K3 surface of Picard rank $0$ with $\alp_X = 0$. In this section we will show that the Brauer-Manin obstruction is not enough to explain the failure of the integral Hasse principle for schemes of this type.

Note that when the equation $cx+d = 0$ is soluble mod $m$ (e.g. when $c$ is coprime to $m$) then $\X$ has an integral point with $y = 0$. Furthermore, if $ax_0+b=\pm 1$ for some $x_0$ then we have a solution with $x=x_0$ and $z=1$. Otherwise, there don't seem to be any obvious integral points on $\X$. We note that by the same argument $\X$ always has a real point and if $a,c$ are coprime than $\X$ has a $\ZZ_p$-point for every $p$.
%
%Let $Y$ be the blow-up of $\A^2$ at the intersection of $((ax+b)y + m) = 0$ and $cx+d$. Then $Y$ can be identified with the subvariety of $\A^2 \times \PP^1$ given by the equation
%$$ ((ax+b)y + m)t = (cx+d)s $$
%where $(t:s)$ are projective coordinates on $\PP^1$. As in the proof of Theorem~\ref{t:realization} we may identify $X$ with the open subset of $Y$ given by $s \neq 0$. The complement of $X$ in $Y$ is the strict transform of the conic $C_0 \subseteq \A^2$ given by
%$$ (ax+b)y + m = 0 .$$

As in the proof of Proposition~\ref{p:character} we may use the long exact sequence associated to cohomology with compact support to prove that $H^2_c(X_{\ovl{\QQ}},\ZZ/n) \cong \ZZ/n$ with trivial Galois action. By Poincare duality with compact support we get that $H^2(X_{\ovl{\QQ}},\ZZ/n(2)) \cong \ZZ/n$ as well, and hence $H^2(X_{\ovl{\QQ}},\ZZ/n(1)) \cong \ZZ/n(-1)$. Since $\Pic(X_{\ovl{\QQ}}) = 0$ it follows that $H^2(X_{\ovl{\QQ}},\GG_m) = H^2(X_{\ovl{\QQ}},\QQ/\ZZ(1)) \cong \QQ/\ZZ(-1)$, and since $H^3(\QQ,\GG_m) = 0$ the Hochschild-Serre spectral sequence implies that
$$ \Br(X)/\Br(k) = \left(\Br\left(X_{\ovl{k}}\right)\right)^{\Gam_k} = (\QQ/\ZZ(-1))^{\Gam_{k}} \cong \ZZ/2 .$$ 
Let us now exhibit a specific generator. 
Consider the quaternion algebra on $X$ given by
$$ A = \left(-\frac{c(ax+b)}{\Del}, \frac{(ax+b)y + m}{m}\right) $$
Then it is straightforward to check that $A$ is unramified in codimension 1 and hence unramified by purity. Furthermore, the residue of $A$ along the curve $C^0 \subseteq \A^2$ given by $(ax+b)y + m = 0$ is the class of the restriction $-\frac{c(ax+b)}{\Del}|_{C^0}$, which is non-trivial (and in fact a generator of $H^1(C^0,\ZZ/2) \cong \ZZ/2$). We may hence conclude that $A$ is a generator of $\Br(X)$.

Now assume that $a,b,c,d$ are pairwise coprime and such that there exists a prime $q$ dividing both $c$ and $m$ exactly once (and hence $a,b,d,\Del,\frac{c}{m}$ are all units mod $q$). Reducing mod $q$ the function $f$ becomes $\ovl{f}(x,y) = (ax+b)y$ and the function $g$ becomes $\ovl{g}(x) = d$. The equation defining $\X \otimes_{\ZZ} \FF_q$ then becomes
$$ (ax + b)yz = d .$$
with $a,b,d$ non-zero, and so $ax+b,y$ and $z$ are invertible functions on $\X_{\FF_q}$. Direct computation shows that the residue of $A$ mod $q$ is the class $[\frac{c}{\Del m}y] \in H^1(\X_{\FF_q},\ZZ/2)$. It is then clear that the evaluation map 
$$ \ev_A: \X(\ZZ_p) \lrar \ZZ/2 $$
is surjective. This means that $A$ poses no Brauer-Manin obstruction (as noted above our assumptions imply in particular that $\X$ has an integral points everywhere locally).

\begin{proof}[Proof of Theorem~\ref{t:counter}]
Consider the scheme $\X$ given by the equation
$$ ((11x+5)y + 3)z = 3x+1 .$$
Here $a,b,c,d = 11,5,3,1$ and $\Del = -4$. We claim that $\X(\ZZ) = \emptyset$. Indeed, observe that $|11x+5| > |3x+1| + 3$ for every $x \in \ZZ$. Now assume that $(x_0,y_0,z_0)$ was a solution. Then $z_0$ would have to be non-zero and so we would obtain
$$ |11x_0 + 5| > |3x_0 + 1| + 3 > |(11x_0 + 5)y_0 + 3| + 3 \geq |11x_0+5||y_0| $$ 
which implies $|y_0| = 0$. But this is impossible because then $3z_0$ would be equal to $3x_0+1$. It follows that $\X(\ZZ) = \emptyset$ as desired. By the above we also know that $\X(\A_{S})^{\Br(X)} \neq \emptyset$ and hence the integral Brauer-Manin obstruction is trivial. Furthermore, since $\alp_X = 0$ it follows that $X$ is $S$-split.
\end{proof}

%
%\begin{rem}
%As $\alp_X = 0$ we observe that $\X$ is in particular an $S$-split $\ZZ$-scheme in the sense of Definition~\ref{d:S-split} (where $S$ is the set containing only the real place). Furthermore, $\X$ is (geometrically) simply connected. It follows that the failure of the Hasse principle cannot be explained by a lack of a well-behaved compactification, or by appealing to suitable \'etale coverings. In particular, Conjecture~\ref{c:rationally-connected} does not extend to the class of log K3 surfaces.
%\end{rem}

\end{document}